\documentclass[a4paper, 12]{amsart}
\usepackage{graphicx}
\usepackage{amsmath}
\usepackage{authblk}
\usepackage{amsthm}
\usepackage{amssymb}
\usepackage{xcolor}
\usepackage[colorlinks]{hyperref}

\usepackage{tikz}
\usepackage{pgfplots}
\pgfplotsset{compat=1.13}
\usepackage{braids}
\usepackage{amsaddr}

\def\L{{\mathcal{L}}}

\def\A{{\widehat{A}}}
\def\X{{\widehat{X}}}

\newcommand{\mat}[4]{\left(\begin{array}{cc}#1 & #2 \\ #3 & #4 \\
	\end{array}\right)}

\usepackage{geometry}
\newtheorem{theorem}{Theorem}[section]
\newtheorem{corollary}{Corollary}[theorem]
\newtheorem{lemma}[theorem]{Lemma}
\newtheorem{Example}{Example}[theorem]
\newtheorem{proposition}[theorem]{Proposition}
\theoremstyle{definition}
\newtheorem{definition}{Definition}[section]

\theoremstyle{remark}

\begin{document}

\title{Solutions to a system of Yang-Baxter matrix equations}

\author{ Askar Ali M, Himadri Mukherjee}
\address{BITS Pilani K. K. Birla Goa Campus, Goa, India}
\author{ Bogdan D. Djordjevic }
\address{Mathematical Institute of the Serbian Academy of Sciences and Arts, Kneza Mihaila 36, 11000 Belgrade, Serbia}
\thanks{ E-mail addresses:  p20190037@goa.bits-pilani.ac.in, himadrim@goa.bits-pilani.ac.in, \\  bogdan.djordjevic@turing.mi.sanu.ac.rs}
\keywords{Yang-Baxter equation, Matrix equation, System of equations, Nonlinear matrix equations}
\thanks{Mathematics Subject Classification: AMS Classification 2020. Primary: 15A24, 16T25}

\maketitle
\noindent
\begin{abstract}
 In this article, a system of Yang-Baxter-type matrix equations is studied, $XAX=BXB$, $XBX=AXA$, which ``generalizes" the matrix Yang-Baxter equation and exhibits a broken symmetry. We investigate the solutions of this system from various geometric and topological points of view. We analyze the existence of doubly stochastic solutions and intertwining solutions to the system and describe the conditions for their existence.  Furthermore, we characterize the case when $A$ and $B$ are idempotent orthogonal complements. i.e., $A^2 =A, B^2= B, AB = BA =0$. We also completely characterize the set of solutions for $n=2$ using commutative algebraic techniques.
\end{abstract}



\section{Introduction} 
 
 For a given pair of square matrices, $A,B \in M_n(\mathbb{K})$, where $\mathbb{K} =\ \mathbb{R}$ or $\mathbb{C}$,
 \begin{equation}\label{main_equation}
 XAX = BXB \text{ and } XBX = AXA
 \end{equation}
 will be called as a system of Yang-Baxter matrix equations. Yang-Baxter equations are extremely useful equations that have been found to have multiple uses in physics, computer science, and many areas of pure mathematics. This equation arises in the context of physics as, in knot theory, if we take the following braids, which are two among the generators of a braid group, and call the first one $A$ and the second one $B$ as in Figure 1, we see that $ABA=BAB$ as braids in Figure 2.

\begin{figure}[htb]
\begin{center}
$A=$ \hspace{.2cm} \begin{tikzpicture}
\braid[number of strands=3] (braid) a_1 ;
\end{tikzpicture}
\hspace{2cm}
$B=$ \hspace{.2cm}\begin{tikzpicture}
\braid[number of strands=3] (braid) a_2 ;
\end{tikzpicture}
\caption{}
\end{center}

\end{figure}
\begin{figure}[htb]
\begin{center}
$BAB=$ \hspace{.2cm}\begin{tikzpicture}
\braid[number of strands=3] (braid) a_2 a_1 a_2 ;
\end{tikzpicture}
\hspace{2cm}
$ABA=$ \hspace{.2cm}\begin{tikzpicture}
\braid[number of strands=3] (braid) a_1 a_2 a_1 ;
\end{tikzpicture}
\caption{}
 \end{center}
\end{figure}
This is not a coincidence that the braids satisfy the equation $AXA=XAX$; it was observed by Artin \cite{Artin} (especially in \cite{Artin2}) in his development of the famous and pioneering theory of braid groups. As a result of the above relations among the generators of a braid group, it makes sense to study the matrix equation that we call YBE.

 While working on the paper \cite{first}, dealing with the Yang-Baxter equation, namely $AXA=XAX$, the authors observed the usefulness of the results that were proved for the coefficient matrix ( i.e., $A$) for the variable $X$ owing to the symmetry of the equation. It was observed that, from the pure matrix equation perspective, it is much better if we deal with the equation $XYX=YXY$, where both $X,Y$ are treated as variables. This equation has the advantage of being homogeneous; thus, many results can be simplified. When specialized, the homogeneous equation, in turn, gives the results for the original YBE. Motivated by this, we will call the equation 
 \[XYX=YXY\]
 as the homogeneous matrix Yang-Baxter equation. Let $V$ be the set of solutions of it over the matrix algebra $M_n(\mathbb{K})$  ( for a suitable field $\mathbb{K}$). We observe the following for the space $V$. Owing to the symmetry of the equations, this has a $\mathbb{Z}_2$ action that flips the coordinates. Continuing our efforts towards simplification, instead of working with the whole $V$, it is much better to work with the $\mathbb{Z}_2$ quotient of it; let us denote the quotient by $\mathfrak{X}$. It will be very interesting to understand the geometry of this object using various tools such as commutative algebra and algebraic geometry. We, taking advantage of the homogeneity of the equation, can look at the projective zero set associated with it, namely, $\{(X,Y) \in \mathbb{P}^{2n^2-1} \, \mid \, XYX=YXY\}$. Let us call this variety as $\mathfrak{Y}$. It is easily observed that this is not an irreducible subset of the projective space, owing to the trivial components $\{(X,0) | X \neq 0 \}$ and $\{(0, )| X \neq 0 \}$, where $0$ represents the zero matrix of an appropriate dimension. So instead, we can look at the projective algebraic set $\{(X,Y) \in \mathbb{P}^{n^2-1} \times \mathbb{P}^{n^2-1} \, \mid \, XYX=YXY\}$. This gets rid of the problems arising out of the trivial solutions of the YBE, and further, if we take a quotient by the action of $\mathbb{Z}_2$, denoted as $\mathfrak{Z}$, we have a much larger chance of asking the following, \[\textit{when is this an irreducible projective variety?}\] 
 This question is interesting, and the authors believe it has a good chance of getting interesting answers. The reason for believing so is that in the paper \cite{manifold}, L. Lu has given the connected components of the space of solutions in some special cases, which turn out to have linear structures. 
 
 It was also observed that towards obtaining good results in the YBE, the main hindrance was the trivial solution $X=A$; this was also observed in the other articles as well, say in \cite{manifold}. This trivial solution gives rise to the lack of simplicity of the solutions of the equation. The homogeneous version of the equation does not solve this problem. But, if we allow a different matrix $Z$ and twist the equation as $XYX=ZXZ$, we have removed the trivial solution, which gives rise to cluttering the attempts at solving the equation. Now, we have lost the beautiful symmetry of the equation, so we have to look at its twin pair $XZX=YXY$. This system of the Yang-Baxter type equation will be called a system of YBE. The authors understand that it is not really a system of YBE as the authors in \cite{first} showed that a true system is just equivalent to a single Yang-Baxter equation, so we have decided to give the honour of being a system of YBE to this system of equations. 
 
  This system of equations, when specialized for $A=B$, gives back the original YBE and gives freedom from difficulties regarding trivial solutions. In this paper, we will show that the solutions are not far from the single Yang-Baxter matrix equation even though we are considering the system. Indeed, in Lemma \ref{block-form} below, we show that a solution of the system is a particular solution to a Yang-Baxter equation. In other words, the matrix $X$ solves the system \eqref{main_equation} if and only if the matrix $\X:=\mat{0}{X}{X}{0}$ solves the YBE $\widehat{A} \X\widehat{A} = \X\widehat{A}\X$, where $\widehat{A} = \begin{pmatrix}
0 & A\\
B & 0
\end{pmatrix} $.
  This, in our thoughts, is a very interesting result, as the symmetry of the original YBE is broken, yet it contains a large amount of information regarding the solutions of YBE.

 The first advantage that we obtain from freeing the equation of the symmetry is the possibility of investigating the equation when the pair $A,B$ are not only idempotents but also orthogonal, a condition that has no non-trivial parallel in the single Yang-Baxter equation. Also, in the single YBE scenario, it was observed that the nilpotent coefficient matrix case is the most interesting and important. A generalization of this for the system is the case that $A,B$ are zero-divisors (i.e., $AB=BA=0$). So, it makes ample sense to investigate the system of YBE when the coefficients satisfy
 \[AB=BA=0.\]
 
 Yang Baxter equations are extremely useful fundamental equations in mathematics and physics. This equation has shown up in unexpected forms in many parts of mathematics and mathematical physics. With the advent of quantum computers and quantum computing in general, the Yang Baxter equations have come into new focus as they are used to design quantum gates for a quantum computer. There are numerous articles exploring this area, namely, \cite{YBE-CS2}, \cite{YBE-CS} and \cite{YBE-CS3} in the areas of neural networks. 

  It has been proved that the variety of commuting nilpotent matrices are irreducible \cite{Commuting_nilp-irr, Commuting-nilp-irr2}. A number of geometric papers concerning the YBE have come up recently, namely the papers \cite{path} and \cite{manifold}. Also, a recent paper investigated the case for the coefficient matrix $A$ to be a permutation matrix and found the solutions in the convex hull of the permutation matrices (i.e., the stochastic matrices) \cite{BDD}. Such geometric understanding of the YBE uses many geometric and topological tools, such as the fixed point theory.
 A comprehensive and systematic analysis of the solution set to the YBE $AXA=XAX$ has been done by Din\v ci\'c and Djordjevi\'c in \cite{path} and \cite{RLT}. Huang et al. \cite{Diag} explored all non-commuting solutions of the Yang-Baxter matrix equation for a class of diagonalizable matrices, while Adam et al. \cite{idempotent} considered the equation for involutive matrices. In addition, Saeed Ibrahim Adam et al. \cite{Tridempotent} provided solutions to the Yang-Baxter-like matrix equation when $A^3=A$. Similarly,  Kumar et al., in \cite{numerical}, provided explicit solutions for the singular Yang-Baxter-like matrix equation and their numerical computation. In addition, the Yang-Baxter-like matrix equation has been investigated for diagonalizable coefficient matrices in \cite{two_eigenvalue, Diag, Diag2, diagonaliazable-simple, 3eigenvalue, ortho,idempotent2}. Many articles explored the solution set for the coefficient matrix having rank restrictions, such as in \cite{rank1,rank2.1,rank2.2}.

Furthermore, there has been significant interest in set-theoretic solutions of the Yang-Baxter equation. Castelli et al. \cite{set_theoretic} investigated the indecomposable involutive set-theoretic solutions of the Yang-Baxter equation of prime-power size, while Lebed and Vendramin \cite{structure_set} focused on the structure groups of set-theoretic solutions to the Yang-Baxter equation. The study of these equations was motivated by Drinfeld's work on unsolved problems in quantum group theory \cite{Drinfeld}.
\subsection{Structure of the article} In the ``Preliminaries" section, we recall some basic results and prove a few small lemmas, which are used in the following sections. In the third section, we investigate the spectral relation between the solutions to \eqref{main_equation} and coefficient matrices $A$ and $B$. Cases are discussed when both coefficient matrices are non-singular, when one is singular and the other is non-singular, and when both are singular. A nonzero solution $X$ to \eqref{main_equation} is said to be an intertwining solution if $AX = XB$ and $BX = XA$. In the same section section, we describe intertwining solutions to \eqref{main_equation} and have drawn conditions for the existence of such solutions. In section four, we have investigated the existence of doubly stochastic solutions to the system using topological tools like Brouwer's fixed point theorem. Two square matrices $A$ and $B$ are called idempotent orthogonal complements, when $A^2 = A, B^2 = B,$ and $AB = BA = 0$. Section five gives a complete characterization of solutions to \eqref{main_equation} when $A$  and $B$ are idempotent orthogonal complements. In the final section, we have found solutions to the system when the coefficient matrices are of dimension 2, using Gr$\ddot{o}$bner basis techniques.

\section{Preliminaries}
In this section, we recall some of the established results for the Yang-Baxter type matrix equation $AXA\ =\ XAX$, which will help the reader to draw a general picture of the solution set of the matrix equation and to understand the following sections as well.
\begin{definition}
    Let $Sol(A)$ denote the set of all solutions to the Yang-Baxter matrix equation, $AXA = XAX$, where $A \in M_n(\mathbb{K})$. 
\end{definition}
\begin{lemma}\cite{first} (see also \cite{RLT})
 For any $g \in GL_n(\mathbb{K})$, $g Sol(A) g^{-1} = Sol(gAg^{-1})$.
\end{lemma}
\begin{theorem}\cite{first}
If $A=\lambda I+B$ is a Jordan block, then we have the following
\begin{enumerate}
    \item If $\lambda \neq 0$ and $det(X) \neq 0$, then $X \simeq A$.
    \item If $\lambda \neq 0$ and $det(X)=0$, then $X=0$.
    \item If $\lambda =0$, then the YBE $AXA=XAX$ do not have any invertible solution X.
\end{enumerate}
\end{theorem}
\begin{lemma}\cite{first}
When $A= \begin{pmatrix}
A_1&0&\dots &0\\
0&A_2&\dots &0\\
.\\
.\\
.\\
0&0&...&A_m
\end{pmatrix}$ block diagonal matrix, with $A_i$ being square matrix of order $n_i$, $X= \begin{pmatrix}
X_1&0&\dots &0\\
0&X_2&\dots &0\\
.\\
.\\
.\\
0&0&...&X_m
\end{pmatrix}$, is a solution to the YBE with coefficient matrix $A$, where $X_i$ is a solution for  the YBE  $A_iYA_i=YA_iY$.
\end{lemma}
\begin{theorem}[Sylvester, \cite{sylvester-original}]\label{sylvester}
     Given matrices $ A\in K^{n\times n}$ and $ B\in K^{m\times m}$, the Sylvester equation $AX+XB=C$ has a unique solution $ X\in  K^{n\times m}$ for any $ C\in K^{n\times m}$, if and only if $A$ and $-B$ do not share any eigenvalue.
\end{theorem}
\begin{theorem}\cite{gen.sylvester}\label{gen.sylvester}
The matrix equation  $AXB-CXD = E$,  where, $A, C \in \mathbb{R}^{m\times m}$, and $D, B\in  \mathbb{R}^{n \times n}$ has a unique solution $X \in \mathbb{R}^{n\times m}$ if and only if,
\begin{itemize}
    \item[i)] $A - \lambda C$ and $D- \lambda B$ are regular matrix pencils, and
\item[ii)] their spectra, $\sigma(A,C)$ and $\sigma (D,B)$ have an empty intersection.
\end{itemize}
\end{theorem}
\begin{definition}\cite{grobner2}
Let $T^n =\ \{x^{a}=x_1^{a_1}x_2^{a_2}\dots x_n^{a_n}/ a_i \in \mathbb{N}\}\ \subset \mathbb{K}[x_1,\dots ,x_n]$. We define the lexicographical order on $T^n$ with $x_1\ >\ x_2\ > \dots >\ x_n$ as follows:
\[ \text{ For } a\ =\ ( a_1 , \dots , a_n ) ,\  b\ =\ ( b_1 , \dots ,b_n ) \in\ \mathbb{N}^n\]
\begin{equation}
  x^a < x^b \Leftrightarrow 
  \begin{cases}
  \mbox{the first coordinates from the left } a_i \mbox{ and }b_i \mbox{ in }a \mbox{ and }b, \\
  \mbox{ which are different, satisfy } a_i<b_i.
  \end{cases}
\end{equation}
So, in the case of two variables $x_1$ and $x_2$, we have
\[1 < x_2 < x_2^2 < x_2^3 < \dots < x_1 < x_2x_1 < x_2^2x_1 < \dots <x_1^2 <\dots \] 
\end{definition}
\begin{definition}\cite{grobner2}
A set of non-zero polynomials $G = \{ g_1 ,g_2,\dots,g_t \}$ contained in an ideal $I \subset \mathbb{K}[x_1,\dots,x_n]$, is called a Gr$\ddot{o}$bner basis for I if for all $f \in I$,  such that $f \neq 0$, there exists $i \in \{ 1 , \dots ,t\}$, such that $lp(g_i)$ divides $lp(f)$, where $lp(f)$ is the leading power product of $f$.
\end{definition}
Using the term order as defined above, one proceeds to perform a ``multivariable division" on the set of polynomial generators of the given ideal $I$.  At the same time, we pick up any ``remainder" that is obtained through this process. A result of this process is a collection of generators that divides any element $f$ of the ideal without leaving any nontrivial ` remainder". This set of generators is called a Gr$\ddot{o}$bner base of the ideal. This ``multivariable" generalization of division algorithm has many applications all throughout mathematics. One of the many applications of Gr$\ddot{o}$bner basis is in the field of elimination theory, in which one ``eliminates" a number of variables from a given set of equations. Once the variables are eliminated, the resulting equations with fewer variables may be solved easily in some situations. 
\begin{theorem}\cite{grobner}
Let $\mathbb{k}$ be a field, and let I be an ideal in $\mathbb{k}[x_1, \cdots, x_n]$.
Let G be a Gr{\"o}bner basis for I with respect to the lexicographic order with $x_1 > x_2 > \dots > x_n$. Then for any $1 \leq t \leq n$, 
$G\  \cap\  \mathbb{k}[x_t, \ldots , x_n]$ is a Gr{\"o}bner basis for the ideal $I\ \cap\ \mathbb{k}[x_t, \ldots , x_n]$.
\end{theorem} 
The above theorem is a key tool in finding solutions in many cases in this paper. We have used Macaulay2 \cite{M2} for computing the Gr{\"o}bner basis.
\section{Spectral Solutions to the System of Yang-Baxter Type Matrix Equations  }

As previously mentioned, the system \eqref{main_equation} is closely connected to a single YBE via the following Lemma:
\begin{lemma}\label{block-form} The matrix $X$ solves the system \eqref{main_equation} if and only if the matrix $\X:=\mat{0}{X}{X}{0}$ solves the equation 
\begin{equation}\label{YBE}
 \widehat{A} \X\widehat{A} = \X\widehat{A}\X,   
\end{equation}
 where $\widehat{A} = \begin{pmatrix}
0 & A\\
B & 0
\end{pmatrix} $.
\end{lemma}
\begin{proof}
    We can rewrite \eqref{main_equation} in the block form as follows.

\[\begin{pmatrix}
0 & A\\
B & 0
\end{pmatrix} 
\begin{pmatrix}
0 & X\\
X & 0
\end{pmatrix} 
\begin{pmatrix}
0 & A\\
B & 0
\end{pmatrix} = 
\begin{pmatrix}
0 & X\\
X & 0
\end{pmatrix}
\begin{pmatrix}
0 & A\\
B & 0
\end{pmatrix}
\begin{pmatrix}
0 & X\\
X & 0
\end{pmatrix}.\]
 This gives the result. Note that these are some special solutions for the Yang-Baxter equation 
 \begin{equation}\label{ybe}
   \widehat{A} \widehat{Y}\widehat{A} = \widehat{Y}\widehat{A}\widehat{Y}. 
 \end{equation}
\end{proof}
This transition actually reduces the initial problem: instead of finding all solutions to \eqref{ybe} (which remains an open problem), we restrict to finding all solutions which have the constant anti-diagonal form $\mat{0}{X}{X}{0}$. This relation allows us to generalize numerous results proved for the original YBE (obtained in papers \cite{path,RLT,DR,BDD,first}) to our system \eqref{main_equation}. In this section, we proceed to prove some of those results.

\begin{theorem}\label{l2.13}
    Let $Sol(A,B)$ denote the set of all solutions to \eqref{main_equation}. For any $g \in Gl_n(\mathbb{K})$, $Sol(gAg^{-1},gBg^{-1}) = gSol(A,B)g^{-1} $.  
\end{theorem}
\begin{proof}
    Let $X \in Sol(A,B)$. This happens if and only if,
    \begin{equation*}
        \begin{aligned}
         gXg^{-1}gAg^{-1}gXg^{-1} &= gBg^{-1}gXg^{-1}gBg^{-1}\\
             \text{and }\\
              gXg^{-1}gBg^{-1}gXg^{-1} &= gAg^{-1}gXg^{-1}gAg^{-1}\\
        \end{aligned}
    \end{equation*}
This gives $ gXg^{-1} \in Sol(gAg^{-1},~gBg^{-1}).$
\end{proof}
As a result of the above Theorem \ref{l2.13}, if $A$ and $B$ are simultaneously diagonalizable, then any solutions to \eqref{main_equation} can be obtained from the following equivalent system of equations.
\[YJ_AY = J_BYJ_B \text{ and } YJ_BY = J_AYJ_A,\] where $J_A, J_B$ denote Jordan canonical forms of $A$ and $B$ respectively. Further, if $A$ and $B$ are simultaneously diagonalizable and similar. Then \eqref{main_equation} is equivalent to the Yang-Baxter equation: $YJ_AY = J_A Y J_A$.
\begin{lemma}\label{2.3}
Let $\mathcal{X}$ be set of all simultaneous solutions to \eqref{main_equation} and to the equation $CXC=\ XCX$, where $C:=\ A+B$. Further, assume the pencils $A+\lambda B$ and $A-\lambda B$ are regular matrix pencils. Then, the set $\mathcal{X}$ contains only a unique solution ($X=0$) if and only if the spectra of the pencils have an empty intersection.
\end{lemma}
\begin{proof}
If $X_0 \in \mathcal{X}$, then we have 
$X_0AX_0 = BX_0B$ and $X_0BX_0 = AX_0A$. Adding these equations gives, 
\begin{align}
    X_0(A+B)X_0 = (A+B)X_0(A+B)-(AX_0B+BX_0A).
\end{align}
Since we have $XCX = CXC$, this implies $AX_0B+BX_0A =\ 0$.\\
Note that the above equation, $AX_0B+BX_0A =0 $, is similar to the generalized Sylvester's equation. The rest follows from the generalized Sylvester's equation Theorem  \ref{gen.sylvester} (unique solution case). 
\end{proof}
\begin{theorem}
Let $\mathcal{X}$ be set of all simultaneous solution to \eqref{main_equation} and $CXC=\ XCX$, where $C\ =\ A+B$. The following are true.
\begin{itemize}
    \item[(a)] Let one of the coefficient matrices be non-singular (without loss of generality, let us assume $A$ is non-singular), and the pencils $A+\lambda B$ and $A-\lambda B$ are regular. If $\mathcal{X}$ contains a non-trivial solution, then $BA^{-1}$ has a pair of eigenvalues $(\mu_1,\mu_2)$, such that their sum is zero.
    \item[(b)] If $A$ and $B$ are invertible, and for some $X \in \mathcal{X}$, $(X,A,B)$ commute pairwise, then $X = 0$.
\end{itemize} 
\end{theorem}
\begin{proof}
\begin{itemize}
    \item[(a)]
By Theorem \ref{gen.sylvester} and Lemma \ref{2.3}, if \eqref{main_equation} has a non-trivial solution, then $\sigma(A,B) \cap \sigma(A, -B) $ is non-empty. Let's say $\lambda_0 \in \sigma(A, B) \cap \sigma(A, -B) $. Then,
\[\text{det}(A+\lambda_0 B) = 0 \text{ and } \text{det}(A-\lambda_0 B) = 0\]
\[ \Rightarrow \text{det}(I+\lambda_0 BA^{-1}) = 0 \text{ and } \text{det}(I-\lambda_0 BA^{-1}) = 0 \]
Which means, there exist a $\mu_1, \mu_2 \in \sigma(BA^{-1})$, such that $1+\lambda_0 \mu_1 = 0$ and $1-\lambda_0 \mu_2 = 0$. Now, since $A$ is invertible, $\lambda = 0$ can not be an eigenvalue for the pencils $A+\lambda B$ and $A-\lambda B$. Which, along with the last pair of equations imply $\mu_1 + \mu_2 = 0$.
\item[(b)] If $X,A,B$ are mutually commutative, and $A, B$ are invertible, then by the above lemma, we have \begin{align*}
  AXB+BXA &= 0\\
 \implies 2ABX &= 0
\end{align*}
Which gives, $X=0$.
\end{itemize}
\end{proof}
Similar conditions can be drawn when $B$ is invertible, as the equation $AX_0B+BX_0A =0 $ has symmetry with respect to $A$ and $B$. In that case, corresponding pencils are $B+\lambda A$ and $B-\lambda A$. Also, note that when both $A,B$ are singular, the necessary condition for the existence of a non-trivial solution is trivially satisfied as $0 \in \sigma(A,B)  \cap \sigma(A, -B) $.

\begin{lemma}
Let $X$ be a solution to \eqref{main_equation}. Then for any $\alpha \in \mathbb{K}$, $X+ \alpha C$ is a solution to \eqref{main_equation} for some $C \in M_n(\mathbb{K})$, if $AC=CA=0$ and $BC=CB=0$.
\end{lemma}
\begin{proof}
We have,
\[ A(X+ \alpha C)A = (X+ \alpha C)B(X+ \alpha C) \]
\[ \text{ and }\]
\[ B(X+ \alpha C)B = (X+ \alpha C)A(X+ \alpha C),\]
is true, if $AC=CA=0$ and $BC=CB=0$.
\end{proof}
\begin{proposition}
Let $X_1,X_2$ be solution to \eqref{main_equation}. Then $X_1+X_2$ is a solution to \eqref{main_equation}, if and only if $X_1AX_2 +X_2AX_1\ =\ 0$ and $X_1BX_2 +X_2BX_1\ =\ 0$.
\end{proposition}
\begin{proof}
\begin{equation*}
    \begin{aligned}
      \Big( A(X_1+X_2)A &= (X_1+X_2)B(X_1+X_2)\\
\text{ and } \\
B(X_1+X_2)B &= (X_1+X_2)A(X_1+X_2)\Big),\\
\iff \Big(AX_1A + AX_2A &= X_1BX_1 +X_1BX_2+X_2BX_1+X_2BX_2\\
\text{ and }\\
BX_1B + BX_2B &= X_1AX_1 +X_1AX_2+X_2AX_1+X_2AX_2\Big),\\
\iff (X_1AX_2+X_2AX_1 &= 0 \text{ and } X_1BX_1 +X_1BX_2\ = 0). 
    \end{aligned}
\end{equation*}

\end{proof}
\begin{definition}
A solution $X$ is commutative if $[A,X]=[B,X]=0$. Otherwise, the solution $X$ is non-commutative.
\end{definition}
The following result is an extension from the theorem given in \cite{path}, which states for an arbitrary square matrix $A$, there are no isolated nontrivial solutions to the YBE $AXA\ =\ XAX$, which do not commute with $A$. Note that the space of solutions to the system of Yang-Baxter equations gets a subspace topology from the usual topology of $M_n(\mathbb{K})$ ( where $\mathbb{K}$ is $\mathbb{R}$ or $ \mathbb{C}$). Considering this topology, the authors of the papers \cite{path,RLT,manifold} prove a very interesting result about the manifold structure of the space of the solutions, and as a consequence, they get results on the ``isolated" solutions (i.e., the singleton path components of the space of solutions). In the same spirit, we prove the following.
\begin{theorem}
Let $A$ and $B$ commute with each other. Then, any non-commutative solution to \eqref{main_equation} is not isolated in the solution set.
\end{theorem}
\begin{proof}
Let $X_0$ be a non-trivial solution to \eqref{main_equation}, which does not commute with both $A$ and $B$. Then we have, for any $t \in \mathbb{K}$, $e^{-BAt}X_0e^{BAt}$ is also a solution to \eqref{main_equation}. We have, as $A$ and $B$ commute, $A$ commute with $e^{BAt}$. This implies,
\begin{align*}
A(e^{-BAt}X_0e^{BAt})A\ &=\ e^{-BAt}AX_0Ae^{BAt}\\
&= e^{-BAt}X_0BX_0e^{BAt}\\
&= (e^{-BAt}X_0e^{BAt})B(e^{-BAt}X_0e^{BAt}).
\end{align*}
Similarly, we can see that,
\[B(e^{-BAt}X_0e^{BAt})B= (e^{-BAt}X_0e^{BAt})A(e^{-BAt}X_0e^{BAt})\] 
This gives, $e^{-BAt}X_0e^{BAt}$ is a solution to \eqref{main_equation}. As a consequence, we can see $X_0$ lies in a path of solutions inside the solution space.
\end{proof}
\subsection{Invertibility of the coefficient matrices}

Spectral properties of the coefficient matrices $A$ and $B$ impose certain spectral restrictions on the solution matrices $X$. Below, we demonstrate how the regularity of the input matrices dictates the behaviour of the solutions.

\begin{theorem}[Necessary conditions for regular solutions]
Let $A$ and $B$ be non-singular matrices. Then:
\begin{itemize}
\item[(a)] Equations in \eqref{main_equation} have a non-singular solution, only if $det(A^3)= det(B^3)$. In particular, if $A$ and $B$ have only real eigenvalues, then there exists a non-singular solution to \eqref{main_equation} only if $det(A) = det(B)$. 
\item[(b)] If the equations in \eqref{main_equation} do indeed have a non-singular solution $X$, then $X^2$ must be similar to $BA$.
\end{itemize}
\end{theorem}
\begin{proof}\begin{itemize}
\item[(a)] We have,
  \[det(AXA) = det(XBX) \text{ and } det(BXB)= det(XAX)\]
  \[\Rightarrow det(A^2)det(X)= det(X^2)det(B)\]
 \[ \text{ and }
  det(B^2)det(X)= det(X^2)det(A).\]
Now, if $X$ is a non-singular solution, then the above equations give, 
\[ det(X) = \frac{det(A^2)}{det(B)} = \frac{det(B^2)}{det(A)}.\]
Which gives, $det(A^3) = det(B^3)$. Further, if $A$ and $B$ have only real eigenvalues, then $det(A) = det(B)$. 
\item[(b)] Let $X$ be an invertible solution for \eqref{main_equation}. 
\[BXB\ =\ XAX\]
\[\Rightarrow B\ =\ X^{-1}B^{-1}XAX\]
\[\Rightarrow BA\ =\ X^{-1}B^{-1}XA A\]
\[\Rightarrow  BA\ =\ X^{-1}B^{-1}X^2BX\]
This gives $X^2$ is similar to $BA$. In other words, $\sigma(X^2) = \sigma(BA)$.
\end{itemize}
\end{proof}
Non-regular solutions also exhibit a certain invariance under the regular coefficient matrices $A$ and $B$:

\begin{lemma} Let $A$ and $B$ be invertible and let $X$ be a singular solution to \eqref{main_equation}. Then $Ker(X)$ is invariant under $A$ and $B$, i.e., $XBv=0$ and $XAv=0$ for every $v\in Ker(X)$.
\begin{proof}
We have $XAXv=0$, therefore $BXBv=0$. Since B is invertible, this gives $XBv=0$. Similarly, we can show that $XAv=0$.
\end{proof}
\end{lemma}
The following result gives an important insight into the solutions when only one input matrix is regular. 
\begin{theorem} \label{eigenvalue-relation}
Let one of the coefficient matrices be singular and the other be non-singular, say $A$ is singular, and $B$ is a non-singular matrix. The following statements are true:
\begin{itemize}
    \item[(a)] Equations in \eqref{main_equation} do not have a non-singular solution. 
    \item[(b)] If $u \in Ker(A)$, then either $u \in Ker(X)$ or $BXu \in Ker(X)$, for any solution $X$ to \eqref{main_equation}.
    \item[(c)] Let $X$ be a solution to \eqref{main_equation} and assume there exists a $\lambda \in \sigma(BA)$ such that $\lambda \notin \sigma(X^2)$. Then $X$ annihilates $E_\lambda(BA)$, the eigen space of $BA$ corresponding to the eigenvalue $\lambda$. Furthermore, the solution $X$ annihilates the generalized eigenspace of $BA$ corresponding to the eigenvalue $\lambda$, $P_\lambda (BA).$
\end{itemize}
\end{theorem}
\begin{proof} Without loss of generality, assume that $A$ is singular and $B$ is non-singular. 
\begin{itemize}
    \item[(a)] Let $v ( \neq 0 ) \in Ker(A)$. Then, $AXAv = 0$, gives $XBXv =0$. Now, if $Xv = 0$, then $v \in Ker(X)$ and X is singular. If $Xv \neq 0$, then as B is non-singular $BXv \neq 0$. That means, $BXv \in Ker(X)$, and X is singular.\\
    \item[(b)] If $u \in Ker(A)$, then we have $AXAu=0$, which gives $BXBu=0$. If $u \notin Ker(X)$, then $BXu (\neq 0) \in Ker(X)$, as B is invertible, for any solution $X$.
    \item[(c)] Let $X$ be a solution to \eqref{main_equation} and let $\lambda$ be such that $\lambda\in\sigma(BA)\setminus\sigma(X^2)$. Respectively, let $v(\neq 0)\ \in\ E_\lambda(BA)$. Then we have
    \begin{align*}
        BAv\ &=\ \lambda v\\
        BXBAv\ &=\ \lambda BX v\\
        XAXAv\ &=\ \lambda BX v\\
        X^2 BXv\ &=\ \lambda BXv\\
    \end{align*}
    Now, since $\lambda \notin \sigma(X^2)$, $BXv$ must be zero. As $B$ is invertible, this gives $Xv\ =\ 0$.\\
    i.e., $XE_\lambda(BA)\ =\ 0$. 
    
    Furthermore, let $\{v_1,v_2,...,v_r\}$ be the canonical basis for the generalized eigenspace $P_\lambda$. Then we have
    \[BAv_1\ =\ \lambda v_1\]
    \[BAv_i\ =\ v_{i-1} + \lambda v_i, \text{ for }i=2,3,..,r.\]
    But from the previous part of the claim, we have $BAv_1\ =\ 0$. Now,
    \begin{align*}
        XBAv_2\ &=\ Xv_1+ \lambda Xv_2\\
        \Rightarrow XBAv_2\ &=\ \lambda Xv_2, \text{  (as $v_1 \in E_\lambda (BA)$, and $Xv_1=0$)}\\
        \Rightarrow BXBAv_2\ &=\ \lambda BXv_2\\
        \Rightarrow XAXAv_2\ &=\  \lambda BXv_2\\
        \Rightarrow X^2BXv_2\ &=\ \lambda BXv_2.
    \end{align*}
    We have $\lambda \notin \sigma(X^2)$ and $B$ invertible, which gives $Xv_2\ =\ 0$. Continuing this process through induction, we have $XP_\lambda(BA)\ =\ 0 $. 
\end{itemize}
\end{proof}

The previous results point out that for any nonzero solution $X$ to \eqref{main_equation}, the matrices $X^2$ and $BA$ must have common eigenvalues. This can be further exploited, as the following analysis shows:
\begin{theorem}
Let $X$ be a solution to \eqref{main_equation}. The following are true.
\begin{itemize}
    \item [(a)] For any polynomial $f \in \mathbb{K}[X]$, $f(X^2)BX = BX f(BA)$. In particular, $\Phi_{BA}(X^2)BX=0$, where $\Phi_{BA}$ is the characteristic polynomial of $BA$.
    \item[(b)]  For every solution $X$ it holds that $e^{X^2}BX\ =\ BXe^{BA}$.
\end{itemize}
\end{theorem}
\begin{proof}
\begin{itemize}
    \item[(a)]We have $XAXA=BXBA$, which implies $X^2BX=BX(BA)$. Now,
    \begin{align*}
        X^4BX &=X^2BX(BA)\\
         &= XAXABA\\
         &= BX (BA)^2
    \end{align*}
    Then by continuing this process, we have $X^{2n}BX=BX(BA)^n$. This implies that, for any polynomial $f$, $f(X^2)BX = BX f(BA)$. In particular, for the characteristic polynomial $\Phi_{BA}$ of $BA$, the later equation can be written as $\Phi_{BA}(X^2)BX=0$.\\
    $(b)$ This part follows immediately from the above proof. 
\end{itemize}
\end{proof}
 For the above Sylvester equation, $e^{X^2}BX = BXe^{BA}$, if $BX \neq 0$, then the matrix $BX$
is a non-zero solution. This implies,  that the
spectra of $e^{X^2}$ and $e^{BA}$ intersect, or, equivalently,
\[\sigma(X^2)\cap (\sigma(BA)+2i\mathbb{Z}\pi) \neq \varnothing.\]
Conversely, if the above sets do not intersect, then $BX = 0$, which forces $X$ to be $0$. 

\subsection{Intertwining Solutions}
A particularly interesting class of solutions is the class of intertwining solutions:
\begin{definition} A nonzero solution $X$ to \eqref{main_equation} is said to be an intertwining solution if $AX=XB$ and $BX=XA$.\end{definition}

These solutions coincide with the class of commuting solutions when $A=B$ and the single YBE $AXA=XAX$ is considered. However, when $A\neq B$ and $X\neq0$ the equality $AX=XB$ and $XA=BX$ is not automatically achieved. 

The existence of intertwining solutions implies that $A$ and $B$ have common eigenvalues. All intertwining solutions can be obtained by intersecting the set of solutions to \eqref{main_equation} with the set of solutions to the homogeneous Sylvester equation $AX=XB$. Below, we provide a much simpler way for intertwining solutions characterization in the case when $A$ and $B$ are regular matrices. Recall the block form \eqref{YBE} from Lemma \ref{block-form}.

\begin{lemma} \label{int}The nonzero matrix $X$ is an intertwining solution to \eqref{main_equation} if and only if the corresponding anti-diagonal matrix $\X=\mat{0}{X}{X}{0}$ is a nonzero commuting solution to \eqref{YBE}. 
\end{lemma}

It is not hard to see that the block-matrix $\A=\mat{0}{A}{B}{0}$ is invertible if and only if $A$ and $B$ are invertible. In that case, the inverse of $\A$ is given as 
$$\A^{-1}=\mat{0}{B^{-1}}{A^{-1}}{0}.$$ 

Thus we proceed to find all commuting solutions to \eqref{YBE} when $\widehat{A}$ is an invertible matrix. Then by Lemma \ref{int}, all those solutions which have the constant anti-diagonal form $\mat{0}{X}{X}{0}$ will indeed be all intertwining solutions for \eqref{main_equation}.

For the given invertible $\A$, its square root $\sqrt{\A}$ denotes the set of all matrices $L$ such that $L^2=\A$. Obviously, the set $\sqrt{\A}$ is non-empty. In \cite{RLT}, Din\v ci\'c proved the following theorem which characterizes all commuting solutions to \eqref{YBE} for the regular coefficient matrix:

\begin{theorem}\label{Dinc}\cite[Theorem 4.2]{RLT} \textbf{(Din\v ci\'c)}
If $\widehat{A}$ is invertible, then
all commuting solutions of \eqref{YBE} are given by the closed-form formula
\begin{equation}\label{comybe}
\widehat{X}=\frac{1}{2}\left(\A+\sqrt{\left(\A\right)^2}\right).
\end{equation}
\end{theorem}

\begin{theorem} Let $A$ and $B$ be regular matrices and let $\A$ be given as above. There exists an intertwining solution to \eqref{main_equation} if and only if there exists a matrix $E$ with the following properties:
\begin{itemize}
\item[(a)] The matrix $A+E$ is regular.
\item[(b)] The equality $A^{-1}EB=BEA^{-1}$ holds.
\item[(c)] The matrix $E$ solves the equation
\begin{equation}
\label{eqE}
(2A+E-2B)+(2A+E-B)EA^{-1}=0.\end{equation}
\end{itemize}
In that case, corresponding to each $E$, a unique intertwining solution $X_E$ to \eqref{main_equation} is defined by the following equivalent system of equations:
\begin{equation}
\label{intertwE}2X_E:=2A+E=B+(A+E)^{-1}AB=B+BA(A+E)^{-1},
\end{equation}
where $E$ is an arbitrary matrix with properties $(a)-(c)$.
\end{theorem}

\begin{proof} By Lemma \ref{int} there exists an intertwining solution $X$ to \eqref{main_equation} if and only if the corresponding block matrix $\mat{0}{X}{X}{0}$ is a nonzero commuting solution to \eqref{YBE}. On the other hand, Theorem \ref{Dinc} yields that all commuting solutions to \eqref{YBE} are given by the expression \eqref{comybe} above. Therefore, the nonzero matrix $X$ is an intertwining solution to \eqref{main_equation} if and only if $\mat{0}{X}{X}{0}$ belongs to the set 
$$\frac{1}{2}\left(\A+\sqrt{\left(\A\right)^2}\right).$$
First note that 
$$\left(\A\right)^2=\mat{AB}{0}{0}{BA}.$$
Let $L=\mat{L_1}{L_2}{L_3}{L_4}$ be an arbitrary member from the set $\sqrt{\left(\A\right)^2}$. Then $L^2=\left(\A\right)^2$, that is,
$$\mat{L_1^2+L_2L_3}{L_1L_2+L_2L_4}{L_3L_1+L_4L_3}{L_3L_2+L_4^2}=\mat{AB}{0}{0}{BA}.$$
In other words, all such matrices $L$ solve the following system of equations
\begin{equation}
\label{Lsys}
\begin{cases}
L_1^2+L_2L_3=AB,\\
L_1L_2+L_2L_4=0,\\
L_3L_1+L_4L_3=0,\\
L_3L_2+L_4^2=BA.
\end{cases}
\end{equation}
Combining the above observations, we conclude that
$$\mat{0}{X}{X}{0}\in\left\{\left.\frac{1}{2}\mat{L_1}{A+L_2}{B+L_3}{L_4}\right| \textrm{ such that \eqref{Lsys} holds} \right\}.$$
It immediately follows that  $L_1=L_4=0$.
Respectively, the system \eqref{Lsys} reduces to $L_2L_3=AB$, $L_3L_2=BA$ and
$$\mat{0}{X}{X}{0}\in\left\{\left.\frac{1}{2}\mat{0}{A+L_2}{B+L_3}{0}\right| \ L_2L_3=AB, \  L_3L_2=BA \right\}.$$
Since $A$ and $B$ are invertible, we have $L_2$ and $L_3$ are invertible as well. Notice that $L_2=A\Leftrightarrow L_3=B$ and, in that case, $A=B=X$, which is the trivial case. Similarly, $L_2=B\Leftrightarrow L_3=A$, and in that case the solution to \eqref{main_equation} would be precisely $\frac{1}{2}(A+B)$, which intertwines $A$ and $B$ if and only if $A$ and $B$ commute, and the proof is complete (in that case the matrix $E$ is precisely $E=B-A$). Otherwise, assume there exists a nonzero mat ix $E$ such that $L_2=A+E$. Notice that $E$ must be such that $A+E$ is invertible. In that case, from $L_2L_3=AB$ we have 
 $$L_3=(A+E)^{-1}AB=(A(I+A^{-1}E))^{-1}AB=(I+A^{-1}E)^{-1}A^{-1}AB=(I+A^{-1}E)^{-1}B$$ while from $L_3L_2=BA$ we have $$L_3=BA(A+E)^{-1}=BA((I+EA^{-1})A)^{-1}=B(I+EA^{-1})^{-1}.$$ 
 Therefore, the system $$\begin{cases}
 L_2L_3=AB,\\L_3L_2=BA
 \end{cases}$$ 
 is solved for $L_2=A+E$ and $L_3=B(I+EA^{-1})^{-1}$ if and only if 
 $$BEA^{-1}=A^{-1}EB,$$
 and in that case of  course $$L_3=B(I+EA^{-1})^{-1}=(I+A^{-1}E)^{-1}B.$$
 Returning to our problem, it follows that the matrix $X$ is an intertwining solution to \eqref{main_equation} if and only if 
 \begin{eqnarray}
 \label{antidiag}
 \begin{aligned}
 &\mat{0}{X}{X}{0}\in\\
 &\left\{\left.\frac{1}{2}\mat{0}{A+A+E}{B+(I+A^{-1}E)^{-1}B}{0}\right| \ (A+E)-\textrm{regular matrix}, \  BEA^{-1}=A^{-1}EB \right\}.
 \end{aligned}\end{eqnarray}
Equating the positions (1,2) and (2,1) in the latter gives
$$\begin{aligned}
&2A+E=(I+(I+A^{-1}E)^{-1})B\Leftrightarrow\\
&2A+E=(I+A^{-1}E)(I+A^{-1}E)^{-1}B+(I+A^{-1}E)^{-1}B\Leftrightarrow\\
&2A+E=(2I+A^{-1}E)(I+A^{-1}E)^{-1}B\Leftrightarrow\\
&2A+E=(2I+A^{-1}E)B(I+EA^{-1})^{-1}\Leftrightarrow\\
&(2A+E)(I+EA^{-1})=(2I+A^{-1}E)B\Leftrightarrow\\
&A(I+EA^{-1})+(I+EA^{-1})A(I+EA^{-1})=(2I+A^{-1}E)B\Leftrightarrow\\
&2A+E+2AEA^{-1}+E^2A^{-1}=2B+A^{-1}EB=2B+BEA^{-1}\Leftrightarrow\\
&(2A+E-2B)+(2A+E-B)EA^{-1}=0.
\end{aligned}$$
Any such matrix $E$, for which $A+E$ is invertible and $BEA^{-1}=A^{-1}EB$, and which solves the equation above
$$(2A+E-2B)+(2A+E-B)EA^{-1}=0$$
defines one particular intertwining solution to \eqref{main_equation}: $$2X_E:=2A+E=B+(I+A^{-1}E)^{-1}B=B+(A+E)^{-1}AB=B+BA(A+E)^{-1}.$$
Similarly, every intertwining solution $X$ to \eqref{main_equation} implies that $\mat{0}{X}{X}{0}$ is a commuting solution to \eqref{YBE}, thus it is contained in the general formula \eqref{comybe}, which written down reduces to the anti-diagonal family \eqref{antidiag} above. Equating the positions (2,1)  and (1,2) gives the sought properties for the matrix $E$.
\end{proof}

\section{Doubly-stochastic solutions}

Permutation and doubly-stochastic matrices play a crucial role in braid group representations, and are closely connected to the original YBE. In this section we analyze the initial system \eqref{main_equation} and search for doubly-stochastic solutions. As it turns out, this is scenario where a system of YBEs behaves quite differently than a single YBE.\\

We start by revisiting some results regarding the single YBE $AXA=XAX$, where $A$ is an invertible matrix. Recall Brouwer fixed point theorem (see e.g. \cite{BEMHMYA}):

\begin{theorem}[Brouwer] Let $\Omega$ be a convex and compact set in the Euclidean space $\mathbb{R}^n$ and let $f:\Omega\to\Omega$ be a continuous mapping in the respective topology. Then, $f$ has a fixed point in $\Omega$.
\end{theorem}

Recall that the polytope of square doubly stochastic matrices (d. s. matrices for short) is a convex and compact set $\mathbb{R}^{n \times n}$, and the permutation matrices are precisely its vertexes (its extreme points). Respectively, any doubly-stochastic matrix can be expressed as a convex combination of permutation matrices. Conversely, if $A$ and $A^{-1}$ are doubly-stochastic, then $A$ is a permutation matrix.

The above discussion requires distinguishing two cases for the invertible coefficient matrix $A$: The case where both $A$ and $A^{-1}$ are doubly-stochastic matrices, and the case where one of them is a d. s. matrix while the other one (out of $A$ and $A^{-1}$) is not a d. s. matrix. These two separate cases have been studied in \cite{BDD} and \cite{DR}, respectively.

\begin{theorem}[Djordjevi\'c]\label{DSDjordj}\cite{BDD} Let $A$ be a permutation matrix. There exists a nontrivial doubly-stochastic solution to $AXA=XAX$ if and only if $A$ has at least one fixed point (i.e., its main diagonal contains at least one 1).\end{theorem}

Mathematically speaking, the equation $AXA=XAX$ is equivalent to $AX=A^{-1}(AX)^2A^{-1}$, which is suitable for the Brouwer's fixed point theorem. This observation was exploited by Ding and Rhee in \cite{DR}, where they obtained the following:

\begin{theorem}[Ding-Rhee]\cite{DR} Let $A$ be an invertible matrix, such that $A^{-1}$ is a doubly-stochastic matrix. Then the YBE $AXA=XAX$ has a doubly-stochastic solution. If $A$ is in addition not a doubly-stochastic matrix itself, then the aforementioned solution is non-trivial.
\end{theorem}

Returning to our initial system \eqref{main_equation}, assume $A$ and $B$ are distinct invertible matrices and recall the block-form \eqref{YBE} from Lemma \ref{block-form}:
$$\widehat{A}=\mat{0}{A}{B}{0}.$$
Notice that $\A$ is a permutation matrix if and only if $A$ and $B$ are permutation matrices as well. In addition, \eqref{YBE} has a doubly-stochastic anti-diagonal solution $\mat{0}{X}{X}{0}$ if and only if the matrix $X$ is a doubly-stochastic solution to \eqref{main_equation}. We proceed to examine the existence of d. s. solutions to \eqref{main_equation}  under these conditions.

\begin{definition}
    A matrix $A= (A_{ij}) \in M_{n\times m}(\mathbb{R})$ is called point-wise non-negative, if for each $i,j$, $A_{ij} \geq 0$.
\end{definition}
\begin{lemma}\label{d.s.product}
    Let $A$ be a doubly stochastic matrix, and for some point-wise non-negative matrix $D$, $DA$ and $AD$ are doubly stochastic. Then, $D$ is a doubly stochastic matrix. 
\end{lemma}
\begin{proof}
    Let $M:= AD$. Then, 
    \[M_{ij} = \sum_{k=1}^n A_{ik}D_{kj}.\]
    Now, the $j^{th}$ column sum of the doubly stochastic matrix $M$ is given by,
    
        \[ \sum_{i=1}^n M_{ij} =~ 1\]
         \[ \implies   \sum_{i=1}^n(\sum_{k=1}^n A_{ik}D_{kj}) =1\]
         \[\implies    D_{1j}(\sum_{i=1}^nA_{i1})+ D_{2j}(\sum_{i=1}^nA_{i2}) + \dots + D_{nj}(\sum_{i=1}^nA_{in}) = 1\]
         \[\implies  \sum_{i=1}^n D_{ij} = 1 \]
As $D_{ij} \geq 0$, this gives $D$ is column stochastic. Similarly, from $DA$ being doubly stochastic, we get $D$ is row stochastic. Together, this implies $D$ is doubly stochastic.
\end{proof}

\begin{theorem} Let $A$ and $B$ be two different non-negative invertible matrices, such that $A^{-1}$ and $B^{-1}$ are doubly-stochastic matrices. If at least one of them is a permutation matrix, then there are no doubly-stochastic solutions to the initial system \eqref{main_equation}.

\end{theorem}
\begin{proof}
First, assume that both $A$ and $B$ are permutation matrices. This makes the matrix $\widehat{A}$ to be a permutation matrix. By the discussion above, \eqref{main_equation} has a d. s. solution $X$ if and only if $\mat{0}{X}{X}{0}$ is a d. s. solution to the YBE \eqref{YBE}. However, due to Theorem \ref{DSDjordj}, this is impossible since $\widehat{A}$ has no 1s on its main diagonal.  

Now assume that $B$ is a permutation matrix, while $A$ is not. Then, for $Y = \begin{pmatrix}
       Y_1&0\\
       0&Y_2
   \end{pmatrix}$, where $Y_i \in \Omega_n$, the set of $n \times n$ doubly-stochastic matrices, consider the following equation:
    \[Y= \widehat{A}^{-1} Y^2\widehat{A}^{-1}.\]
    By the Brouwer's fixed point theorem, the above equation has a fixed point, say, $Y' =  \begin{pmatrix}
       Y_1'&0\\
       0&Y_2'
   \end{pmatrix} $. This gives rise to a solution $X_0$ to the initial system of equations, if and only if $X_0 = A^{-1}Y'_1 = B^{-1}Y'_2$, and $X_0 = Y'_1  A^{-1}= Y'_2 B^{-1}.$ By construction, $X_0$ is a d. s. solution and 
   $Y_1' = AB^{-1}Y_2'$ and $Y_1' = Y_2'B^{-1}A$.  Lemma \ref{d.s.product} implies that $A$ is a d. s. matrix as well. Since $A$ and $A^{-1}$ are d.s, $A$ must be a permutation matrix. This contradicts the initial assumption that $A$ is not a permutation matrix.  
\end{proof}

Since the above theorem disproves the existence of d. s. solutions if at least one of the matrices is a permutation, we provide an alternative way to explore the solutions.

\begin{theorem}[Fixed point approach]\label{nec}  Let $A$ and $B$ be invertible matrices. The following statements are true:
\begin{itemize}
\item[(a)] If $X_0$ solves \eqref{main_equation} then $AX_0$ solves the fixed point problem
\begin{equation}
\label{Yfixed}
Y=A(AB)^{-1}Y^2(AB)^{-1}A.
\end{equation}
while $BX_0$ solves the fixed point problem \begin{equation}\label{Zfixed}
Z=B(BA)^{-1}Z^2(BA)^{-1}B.
\end{equation}
\item[(b)] Let $Y_0$ and $Z_0$ be two solutions to the fixed point problems \eqref{Yfixed} and \eqref{Zfixed} respectively. If $A^{-1}Y_0=B^{-1}Z_0=:X_0$, then $X_0$ solves the initial system \eqref{main_equation}.
\end{itemize}
\end{theorem}
\begin{proof} $(a)$ We have $XB=B^{-1}XAX$ and $XA=A^{-1}XBX$, therefore
$$\begin{aligned}
&XB=B^{-1}(XA)X=B^{-1}A^{-1}XBX^2=\\
&(AB)^{-1}XBX^2=(AB)^{-1}AXAX=(AB)^{-1}(AX)^2\Leftrightarrow\\
&X=(AB)^{-1}(AX)^2B^{-1}A^{-1}A=(AB)^{-1}(AX)^2(AB)^{-1}A\Leftrightarrow\\
&AX=A\left((AB)^{-1}(AX)^2(AB)^{-1}\right)A.
\end{aligned}$$
Substituting $Y=AX$ gives 
\begin{equation}
\label{yfixed}
Y=A(AB)^{-1}Y^2(AB)^{-1}A.
\end{equation}
Analogously, from $XA=A^{-1}XBX=A^{-1}B^{-1}XAX^2$ we get
$$BX=B((BA)^{-1}(BX)^2(BA)^{-1})B,$$
and substituting $Z=BX$ we get 
\begin{equation}\label{zfixed}
Z=B(BA)^{-1}Z^2(BA)^{-1}B.
\end{equation}
$(b)$ Conversely, let $Y$ and $Z$ be solutions to \eqref{Yfixed} and \eqref{Zfixed}, respectively such that $A^{-1}Y=B^{-1}Z=X$. Then 
$$A^{-1}Y=B^{-1}A^{-1}YYB^{-1}A^{-1}A=B^{-1}(A^{-1}Y)YB^{-1}\Leftrightarrow BXB=XY=XAX.$$
Analogously 
$$B^{-1}Z=A^{-1}B^{-1}ZZA^{-1}B^{-1}B=A^{-1}XZA^{-1}\Leftrightarrow AXA=XBX.$$
\end{proof}

\begin{corollary}
Let $A$ and $B$ be different permutation matrices. Let $\mathcal{Y}$ be the set of all doubly-stochastic solutions to the equation \eqref{Yfixed} and let $\mathcal{Z}$ be the set of all doubly-stochastic solutions to \eqref{Zfixed} (both of these sets are non-empty due to the Brouwer's fixed point theorem). Then 
$$(\forall Y\in\mathcal{Y}) (\forall Z\in\mathcal{Z})\quad A^{-1}Y\neq B^{-1}Z.$$
\end{corollary}
\begin{corollary}
Let $A$ and $B$ be invertible matrices, such that $A^{-1}$ and $B^{-1}$ are doubly-stochastic matrices while $A$ and $B$ are not doubly-stochastic matrices. If the system of equations 
\begin{equation}
A^{-1}YA^{-1}=(AB)^{-1}Y^2(AB)^{-1},\quad B^{-1}ZB^{-1}=(BA)^{-1}Z^2(BA)^{-1}
\end{equation}
has a doubly-stochastic solution $(Y_0,Z_0)$ such that $A^{-1}Y_0=B^{-1}Z_0=:X_0$, then the doubly-stochastic matrix $X_0$ is a solution to \eqref{main_equation}.
\end{corollary}

\section{Idempotent Orthogonal Complements}
Let us recall that two square matrices $A$ and $B$ are called idempotent orthogonal complements when $A^2=A, B^2 = B$, and $AB=BA=0$. This section investigates solutions to \eqref{main_equation} when $A$ and $B$ are idempotent orthogonal complements. As $A$ and $B$ are idempotent, their spectra are subsets of $\{0,1\}$, also both are diagonalizable.\\
Let $E_A(\lambda_i)$ denote the eigenspace of $A$ with respect to the eigenvalue $\lambda_i$. Similarly, for $B$, $E_B(\lambda_j)$ denotes the eigenspace of $B$ with respect to the eigenvalue $\lambda_j$. Now, for a non-zero vector $v \in E_B(1)$, we have $Bv=v$. This gives, $ABv=0$, implies $Av=0$. Similarly, any eigenvector for $A$ corresponding to eigenvalue one must belong to $Ker(B)$. Hence, we have the following inclusions.
\[ E_B(1) \subset E_A(0)\]
\[ E_A(1) \subset E_B(0)\]  
Analogously, we can see that any column vector of $B$ belongs to $Ker(A)$, and vice versa. As we discussed above, both $A$ and $B$ are diagonalizable and commute. Hence, they are simultaneously diagonalizable. i.e- There exist an invertible matrix $S \in GL_n(\mathbb{K})$, such that $J_A= SAS^{-1}$ and $J_B= SBS^{-1}$. We say two matrices $A, B \in M_n(\mathbb{K})$ are similar, denoted by $A \simeq B$, if there exists a $P \in GL_n(\mathbb{K})$, such that $A = PBP^{-1}$. Then, as a consequence of Theorem \ref{l2.13}, without loss of generality, we can assume both $A$ and $B$ are in their Jordan-canonical forms. Also, since their spectra are subsets of $\{0,1\}$, we have, for an appropriate $r$, $A \simeq \begin{pmatrix}
I_r&0\\
0&0
\end{pmatrix}$, where $I_r$ is identity matrix of order r, and $0's$ are zero blocks in appropriate dimensions. To avoid the trivial cases of equations in \eqref{main_equation}, let's take $0 < r< n$.
Now, by the above inclusions for eigenspaces of $A$ and $B$, we have $nullity(B) \geq rank(A)$. This leads to, $B \simeq \begin{pmatrix}
0&0\\
0& I_m
\end{pmatrix}$ , such that $n-m \geq r$.\\
Throughout the discussion, $J$ represents the following matrix with appropriate dimensions.
\[\text{ Let } J_2 =\ \begin{pmatrix}
    0&1\\
    0&0
\end{pmatrix} \text{ and }
J = diag(J_2,J_2, \dots, J_2, 0,0, \dots, 0).\] 
\begin{lemma}
    Let $A$, $B$ be two idempotent orthogonal complement matrices, and $X$ be a solution to \eqref{main_equation} for this given pair. Then $AX$ and $BX$ are nilpotent matrices with index at most 2. 
\end{lemma}
\begin{proof}
    We have, \begin{align*}
        AXA &= XBX\\
        BAXA &= BXBX\\
        \implies (BX)^2 &= 0. 
    \end{align*}
Similarly, we get $(AX)^2=0$. 
\end{proof}
\begin{theorem}
Let $A$ and $B$ be two idempotent matrices, which are orthogonal complements and are in their Jordan-canonical forms as considered above, such that $rank(A) = nullity(B) = r$. Then, any solution to \eqref{main_equation} has one of the following forms.
\begin{itemize}
    \item[(a)] $X= \begin{pmatrix}
     0&C\\
     D&0
    \end{pmatrix}$, such that $CD=0$ and $DC = 0$.
    \item[(b)] $X = \begin{pmatrix}
     0&C\\
     D&Y_2
    \end{pmatrix} $, where $Y_2 = UJU^{-1}$, for any $U \in GL_{n-r}(\mathbb{K})$, and for $Z = U^{-1}C,\ W\ =\ DU$, $ZJ=0,\ JW = 0,\ ZW = 0,\ WZ = J$.
    \item[(c)] $X= \begin{pmatrix}
      Y_1&C\\
     D&0
    \end{pmatrix}$,  where $Y_1 = VJV^{-1}$, for any $V \in GL_r(\mathbb{K})$, and for $Z\ =\ V^{-1}C,\ W\ =\ DV$, $JZ\ =\ 0,\ WJ\ =\ 0,\ ZW\ =\ J,\ WZ\ =\ 0$.
    \item[(d)]  $X = \begin{pmatrix}
     Y_1&C\\
     D&Y_2
    \end{pmatrix} $, where $Y_1 = U_1JU_1^{-1}, Y_2 = U_2JU_2^{-1}$ for any $(U_1,U_2) \in GL_r(\mathbb{K}) \times GL_{n-r}(\mathbb{K})$, and for $Z_1\ =\ U_1^{-1}C,\ Z_2\ =\ CU_2,\ W_1\ =\ DU_1,\ W_2\ =\ U_2^{-1} D$, $JZ_1\ =\ 0, \ Z_2J\ =\ 0,\ W_1J\ =\ 0,\ JW_2\ =\ 0,\ Z_1W_1\ =\ J,\ Z_2W_2\ =\ J$.
\end{itemize}
\end{theorem}
\begin{proof}
When $rank(A)= nullity(B)$; i.e $r=n-m$, we have,
\[  A = \begin{pmatrix}
I_r&0\\
0&0
\end{pmatrix}
\text{ and }  B= \begin{pmatrix}
0&0\\
0& I_{n-r}
\end{pmatrix}.\]
Also, let us partition $X$ corresponding to the block forms of $A$ and $B$. Take, $X = \begin{pmatrix}
 Y_1 & C\\
 D & Y_2
\end{pmatrix}$ , where $Y_1$ is $r \times r$, $Y_2$ is $(n-r) \times (n-r)$, C is $r \times n-r$, and D is $r \times n-r$ matrices.
Then, \eqref{main_equation} implies, 
\[ \begin{pmatrix}
I_r&0\\
0&0
\end{pmatrix}  \begin{pmatrix}
 Y_1 & C\\
 D & Y_2
\end{pmatrix}  
 \begin{pmatrix}
I_r&0\\
0&0
\end{pmatrix} = 
\begin{pmatrix}
 Y_1 & C\\
 D & Y_2
\end{pmatrix} 
\begin{pmatrix}
0&0\\
0& I_{n-r}
\end{pmatrix}
\begin{pmatrix}
 Y_1 & C\\
 D & Y_2
\end{pmatrix} 
\]
and 
\[ \begin{pmatrix}
0&0\\
0& I_{n-r}
\end{pmatrix}  \begin{pmatrix}
 Y_1 & C\\
 D & Y_2
\end{pmatrix}  
 \begin{pmatrix}
0&0\\
0& I_{n-r}
\end{pmatrix} = 
\begin{pmatrix}
 Y_1 & C\\
 D & Y_2
\end{pmatrix} 
\begin{pmatrix}
I_r&0\\
0&0
\end{pmatrix}
\begin{pmatrix}
 Y_1 & C\\
 D & Y_2
\end{pmatrix}.  \]
We obtain the following system of equations,
\begin{equation}\label{eq15}
    \begin{aligned}
        Y_1^2=0, Y_2^2 = 0, Y_1C =0, CY_2=0, DY_1=0, Y_2D=0, CD=Y_1, DC=Y_2.
    \end{aligned}
\end{equation}
This gives us, $Y_1$ and $Y_2$ are nilpotent matrices with nipotency $\leq 2$. With the choice of $Y_1$ and $Y_2$, we can determine the remaining blocks $C,D$.\\
\textbf{Case 1:} Let both $Y_1$ and $Y_2$ be zero matrices, then equations in (\ref{eq15}) reduces to finding $C$ and $D$ such that $CD=0 ,DC =0$. Now, if $Y_1 = 0$ and $Y_2$ is non-zero, then $Y_2$ must be similar to J. Say, $Y_2 = UJU^{-1}$, for some $U \in GL_{n-r}(\mathbb{K})$. Then by taking $Z = CU , W = U^{-1}D$, we can rewrite equations in (\ref{eq15}) as the following.
\begin{equation}\label{eq16}
    \begin{aligned}
        ZJ\ =\ 0,\ JW\ =\ 0,\ ZW\ =\ 0,\ WZ\ =\ J.
    \end{aligned}
\end{equation}
Note that, here, the first column of $Z$ and second row of $W$ are zero vectors, and rows of $Z$ are orthogonal to column vectors of $W$.\\
\textbf{Case 2:} Similarly, if $Y_2\ =\ 0$ and $Y_1$ is non-zero, then $Y_1$ must be similar to $J$. Then for some $V \in GL_r(\mathbb{K})$, $Y_1\ =\ VJV^{-1}$. Then by taking $Z\ =\ V^{-1}C,\ W\ =\ DV$, we can rewrite (\ref{eq15}) as follows. \[JZ\ =\ 0,\ WJ\ =\ 0,\ ZW\ =\ J,\ WZ\ =\ 0.\]
\textbf{Case 3:} Now, let both $Y_1$ and $Y_2$ be non-zero. Then, $Y_1 =  U_1JU_1^{-1} $ and $Y_2 = U_2JU_2^{-1} $, for some $(U_1, U_2) \in GL_r(\mathbb{K}) \times GL_{n-r}(\mathbb{K})$. Then, by taking $Z_1 = U_1^{-1}C, Z_2 = CU_2, W_1 = DU_1, W_2 = U_2^{-1} D$, we can rewrite (\ref{eq15}) as the following.
\begin{equation}\label{eq.14}
    \begin{aligned}
        JZ_1 = 0,\  Z_2J = 0,\ W_1J = 0,\ JW_2 = 0,\ Z_1W_1 = J,\ W_2Z_2 = J. 
    \end{aligned}
\end{equation}
Here, the first columns of $Z_2$ and $W_1$ are zero, also the second rows of $Z_1$ and $W_2$ are zero.
\\
\end{proof}

\begin{Example}
Let $A\ = \frac{1}{31}\begin{pmatrix}
 -39&-20 &-50\\
 14&35&10\\
 49&14&66
\end{pmatrix}$, and\\
$B\ = \frac{1}{31}\begin{pmatrix}
 70&20&50\\
 -14&-4&-10\\
 -49&-14&-35
\end{pmatrix}.$ Then clearly, both $A$ and $B$ are idempotent and orthogonal complements, and are simultaneously diagonalizable as follows. For $U = \begin{pmatrix}
 3&1&4\\
 2&3&2\\
 7&2&5
\end{pmatrix} $,
\begin{align*}
UAU^{-1} &=  \begin{pmatrix}
 3&1&4\\
 2&3&2\\
 7&2&5
\end{pmatrix}
\frac{1}{31}\begin{pmatrix}
 -39&-20 &-50\\
 14&35&10\\
 49&14&66
\end{pmatrix}\frac{1}{31}
\begin{pmatrix}
 -11&-3&10\\
 -4&13&-2\\
 17&-1&-7
\end{pmatrix}\\
 &= \begin{pmatrix}
 1&0&0\\
 0&1&0\\
 0&0&0
\end{pmatrix}\ = J_A 
\end{align*}
Similarly, for the same $U$, we can see, $J_B\ = UBU^{-1}= \begin{pmatrix}
 0&0&0\\
 0&0&0\\
 0&0&1
\end{pmatrix}$.\\
Then, a solution $X$ for \eqref{main_equation}, can be found using the equations, $YJ_AY = J_BYJ_B$ and $YJ_BY = J_AYJ_A$, where $Y\ =\ UXU^{-1}$.\\
Now, take $Y\ =\ \begin{pmatrix}
 a&b&c\\
 d&e&f\\
 g&h&i
\end{pmatrix}$. Then, $J_AYJ_A\ =\ YJ_BY$, and $J_BYJ_B\ =\ YJ_AY$ implies the following equations. The solution $Y$ is constituted by the zeros of the following system.
\begin{equation}
    \begin{aligned}
    \begin{matrix}
     f_1= -cg+a& f_2= -ch+b&  f_3= ci \\
 f_4= -fg+d&  f_5 = -fh+e& f_6= fi \\
f_7= gi&  f_8= hi&  f_9= i^2 \\
f_{10}=a^2+bd & f_{11}= ab+be & f_{12}= ac+bf   \\
f_{13}= ad+de & f_{14}=bd+e^2 & f_{15}= cd+ef   \\
f_{16}= ag+dh& f_{17}= bg+eh & f_{18}= -cg-fh+i
    \end{matrix}
    \end{aligned}
\end{equation}
Now, if we calculate the Gr$\ddot{o}$bner basis for ideal $\langle f_1,f_2,...,f_{18} \rangle$ in $\mathbb{C}[a,b,c,...,i]$ with respect to the lexicographic order for $a>b>c>d>...>i$, we get, $\{  a+e-i,\ i^2,\ hi,\ gi,\ fi,\ ei,\ di,\ ci,\ bi,\ fh-e,\ ch-b,\ fg-d,\ eg-dh,\ cg+e-i,\ bg+eh,\ ce-bf,\ cd+ef,\ bd+e^2\}.$\\
From the Gr$\ddot{o}$bner basis, we have $i =\ 0$, gives the block $Y_2= 0$. From remaining equations, we can see that $Y_1^2\ =\ \begin{pmatrix}
 a&b\\
 d&e
\end{pmatrix}^2 = 0$. Also, \[CD\ =\ \begin{pmatrix}
 c\\
 f
\end{pmatrix} \begin{pmatrix}
 g&h
\end{pmatrix}\ = \begin{pmatrix}
 cg&ch\\
 fg&fh
\end{pmatrix}\ =\ \begin{pmatrix}
 a&b\\
 d&e
\end{pmatrix}\ =\ Y_1. \]
Similarly, we can see that $DC=0$, $Y_1C\ = 0$, $DY_1\ = 0$. So, the solution set is precisely, $X\ = U^{-1} \begin{pmatrix}
 Y_1&C\\
 D&0
\end{pmatrix}U$, with $Y_1^2\ =\ 0$, $CD\ =\ Y_1$, $DC\ =\ 0$, $Y_1C\ = 0$, $DY_1\ = 0$. 
\end{Example}

\begin{theorem}
Let $A,B$ be two idempotent matrices, which are orthogonal complements and are in their Jordan-canonical forms as follows, with $rank(A) < nullity(B)$.
\[  A = \begin{pmatrix}
I_r&0\\
0&0
\end{pmatrix}
\text{ and }  B= \begin{pmatrix}
0&0\\
0& B'
\end{pmatrix}, \]
where $B' = \begin{pmatrix}
 I_s&0\\
 0&0
\end{pmatrix}$, $(n-r)\times (n-r)$ matrix, and $s = rank(B)$, such that $1 \leq s < n-r$. Then any solution to \eqref{main_equation} has one of the following forms.\\
\begin{itemize}
    \item[(a)] $X= \begin{pmatrix}
     0&C\\
     D&Y_2
    \end{pmatrix}$, such that $DC=B'Y_2B',~ CB'D=Y_1,~CB'Y_2 =0,~ Y_2B'D = 0,~ Y_2B'Y_2 = 0.$
    \item[(b)] $X = \begin{pmatrix}
     Y_1&C\\
     D&Y_2
    \end{pmatrix} $, where $Y_1 = UJU^{-1}$, for any $U \in GL_r(\mathbb{K})$, and for $Z\ =\ U^{-1}C,\ W\ =\ DU$, $ JZ=0,~ WJ=0,~ DC = B'Y_2B',~ ZB'W=J,~CB'Y_2 =0,~Y_2B'D=0,~Y_2B'Y_2 =0.$
\end{itemize}
\end{theorem}
\begin{proof}
 Let us take X as previously, i.e., $X =  \begin{pmatrix}
 Y_1 & C\\
 D & Y_2
\end{pmatrix}$. Then \eqref{main_equation} implies,
\[ \begin{pmatrix}
I_r&0\\
0&0
\end{pmatrix}  \begin{pmatrix}
 Y_1 & C\\
 D & Y_2
\end{pmatrix}  
 \begin{pmatrix}
I_r&0\\
0&0
\end{pmatrix} = 
\begin{pmatrix}
 Y_1 & C\\
 D & Y_2
\end{pmatrix} 
\begin{pmatrix}
0&0\\
0& B'
\end{pmatrix}
\begin{pmatrix}
 Y_1 & C\\
 D & Y_2
\end{pmatrix} 
\]
and 
\[ \begin{pmatrix}
0&0\\
0& B'
\end{pmatrix}  \begin{pmatrix}
 Y_1 & C\\
 D & Y_2
\end{pmatrix}  
 \begin{pmatrix}
0&0\\
0& B'
\end{pmatrix} = 
\begin{pmatrix}
 Y_1 & C\\
 D & Y_2
\end{pmatrix} 
\begin{pmatrix}
I_r&0\\
0&0
\end{pmatrix}
\begin{pmatrix}
 Y_1 & C\\
 D & Y_2
\end{pmatrix}.  \]
This gives rise to the following equations.
\begin{equation}\label{eq18}
\begin{aligned}
    Y_1^2 = 0,~ Y_1C=0,~ DY_1=0,~ DC=B'Y_2B',\\
    CB'D=Y_1,~ CB'Y_2 =0,~ Y_2B'D = 0,~ Y_2B'Y_2 = 0.
    \end{aligned}
\end{equation}
Further, by partitioning $C,D,$ and $Y_2$ into blocks with dimensions corresponding to $B'$, we have
\[  C = \begin{pmatrix}
 C_1&C_2\\
 C_3&C_4
\end{pmatrix}, D= \begin{pmatrix}
 D_1&D_2\\
 D_3&D_4
\end{pmatrix}, Y_2 = \begin{pmatrix}
 X_1&X_2\\
 X_3&X_4
\end{pmatrix}.\]
Then, we have the following implications.
\begin{equation}\label{eq17}
    \begin{aligned}
    Y_2B'Y_2 = 0 \Rightarrow ~~ X_1^2 =0,~ X_1X_2 = 0,~ X_3X_1 =0,~ X_3X_2=0\\
    CB'Y_2=0 \Rightarrow ~~ C_1X_1 =0,~ C_1X_2 =0 ,~ C_3X_1=0,~C_3X_2 =0\\
    Y_2B'D =0 \Rightarrow ~~ X_1D_1=0,~X_1D_2 =0,~ X_3D_1=0,~X_3D_2=0\\
    DC=B'Y_2B' \Rightarrow ~~ D_1C_1+D_2C_3=X_1,~ D_1C_2+D_2C_4=0,\\
    D_3C_1+D_4C_3=0,~ D_3C_2+D_4C_4=0
    \end{aligned}
\end{equation}
Now, from  \eqref{eq18}, it is clear that $Y_1$ is a nilpotent matrix with $nilpotency \leq 2$. Then, $Y_1$ is either 0 or similar to J.
\textbf{Case 1:} When $Y_1 = 0$,
\[ CB'D = Y_1 =0 \Rightarrow C_1D_1=0,C_1D_2 =0, C_3D_1=0,C_3D_2=0. \]
\textbf{Case 2:} If $Y_1 \neq 0$, then for some $ U \in GL_r(\mathbb{K})$, $Y_1 = UJU^{-1}$. Then by taking $U^{-1}C=Z$ and $DU = W$, we can rewrite (\ref{eq18}) as following,
\begin{equation}\label{eq20}
   JZ=0, WJ=0, DC = B'Y_2B', ZB'W=J, CB'Y_2 =0, Y_2B'D=0, Y_2B'Y_2 =0.
\end{equation}
Further, from the first equation of \eqref{eq17}, it is clear that $X_1$ is nilpotent with $nilpotency \leq 2$. Then $X_1$ is 0 or similar to J. Now, if $X_1 =0$, then we can rewrite $CB'Y_2=0$ as, $\begin{pmatrix}
 0&C_1X_2\\
 0&C_3X_3
\end{pmatrix}\ =\ Y_1$. Also, $Y_2B'D = 0$ can be rewritten, $\begin{pmatrix}
 0&0\\
 X_3D_1&X_3D_2
\end{pmatrix}\ =\ \begin{pmatrix}
 0&0\\
 0&0
\end{pmatrix}$.\\
Also, when $s \geq 2$, and if $Y_1 \neq 0$, then  by taking $Z= \begin{pmatrix}
 Z_1&Z_2\\
 Z_3&Z_4
\end{pmatrix}$ and $W= \begin{pmatrix}
 W_1&W_2\\
 W_3&W_4
\end{pmatrix}$, 
$ZB'W = J$ implies, $Z_1W_2 =0, Z_3W_1 =0, Z_3W_2 =0$, and $Z_1W_1 = J_s$, where $J_s$ is the left-top corner $s \times s$ block of J.
\end{proof}
\begin{Example}
Let $A\ =\ \begin{pmatrix}
 29/2&24&12&33/2\\
 -1/2&-2&-1&-7/2\\
 -43/2&-34&-17&-41/2\\
 9/2&8&4&13/2
\end{pmatrix}$, and \\
$B\ =\ \begin{pmatrix}
 -12&-15&-9&-6\\
 0&0&0&0\\
 20&25&15&10\\
 -4&-5& 3&-2
\end{pmatrix}$.
Then, both $A$ and $B$ are idempotent and orthogonal complements. We can simultaneously diagonalize them using $U =\begin{pmatrix}
 1&2&1&2\\
 3&4&2&1\\
 4&5&3&2\\
 1&6&2&7
\end{pmatrix} $, as follows.
\begin{align*}
   UAU^{-1} &= \begin{pmatrix}
 1&2&1&2\\
 3&4&2&1\\
 4&5&3&2\\
 1&6&2&7
\end{pmatrix}
\begin{pmatrix}
 29/2&24&12&33/2\\
 -1/2&-2&-1&-7/2\\
 -43/2&-34&-17&-41/2\\
 9/2&8&4&13/2
\end{pmatrix}
\begin{pmatrix}
 7&5/2&-3&-3/2&\\
 -2&1/2&0&1/2\\
 -8&-9/2&5&3/2\\
 3&1/2&-1&-1/2
\end{pmatrix}\\
&= \begin{pmatrix}
 1&0&0&0\\
 0&1&0&0\\
 0&0&0&0\\
  0&0&0&0
\end{pmatrix} = J_A.
\end{align*}
Similarly, for the same $U$, we have $J_B\ =\ UBU^{-1}\ = \begin{pmatrix}
  0&0&0&0\\
  0&0&0&0\\
  0&0&1&0\\
  0&0&0&0
\end{pmatrix}.$\\
Then, a solution $X$ for \eqref{main_equation}, can be found using the equations, $YJ_AY = J_BYJ_B$ and $YJ_BY = J_AYJ_A$, where $Y\ =\ UXU^{-1}$.\\
Take $Y\ =\ \begin{pmatrix}
 Y_1&C\\
 D&Y_2
\end{pmatrix}\ =\ \begin{pmatrix}
 a&b&c&d\\
 e&f&g&h\\
 i&j&k&l\\
 m&n&p&q
\end{pmatrix}$.\\
Then the Gr$\ddot{o}$bner basis generated corresponding to the equations from the system $J_AYJ_A\ =\ YJ_BY$ and $J_BYJ_B\ =\ YJ_AY$ is:\\
$\{ a+f-k,\ lp,\ kp,\ jp,\ ip,\ fp,\ ep,\ bp,\ fm-km-en,\ dm+hn,\ cm+gn,\ bm+fn,\ kl,\ gl,\ fl,\ el,\ cl,\ bl,\ k^2,\ jk,\\ ik,\ gk,\ fk,\ ek,\ ck,\ bk,\ gj-f,\ cj-b,\ gi-e,\ fi-ej,\ di+hj,\ ci+f-k,\ bi+fj,\ df-bh-dk,\ cf-bg,\ de+fh,\ ce+fg,\ be+f^2,\ hkn,\ dkn,\ dgn-chn,\ hkm,\ hjm-hin,\ ejm-ein,\ bdg-bch \}$\\

From the above Gr$\ddot{o}$bner basis, we get $k=0$, which gives $a\ = -f$. Then, we have, 
\[Y_2\ =\ \begin{pmatrix}
 X_1&X_2\\
 X_3&X_4
\end{pmatrix} = \begin{pmatrix}
 k&l\\
 p&q
\end{pmatrix} = \begin{pmatrix}
 0&l\\
 p&q
\end{pmatrix}, \text{ with $lp =0$}.\]
Also, $Y_1^2\ = \begin{pmatrix}
 a&b\\
 e&f
\end{pmatrix}^2\ = 0.$ Similarly, by considering $C\ = \begin{pmatrix}
 C_1&C_2\\
 C_3&C_4
\end{pmatrix}\ = \begin{pmatrix}
 c&d\\
 g&h
\end{pmatrix}$ and $D\ = \begin{pmatrix}
 D_1&D_2\\
 D_3&D_4
\end{pmatrix}\ = \begin{pmatrix}
 i&j\\
 m&n
\end{pmatrix}$, we can see,
\[CB'Y_2\ = \begin{pmatrix}
 0&C_1X_2\\
 0&C_3X_2
\end{pmatrix}\ = \begin{pmatrix}
 0&cl\\
 0&gl
\end{pmatrix}.\]
In same passion we have, $Y_2B'D = 0,\ Y_2B'Y_2=0,\ CB'D= Y_1,\ DC = B'Y_2B',\  Y_1 =0,\ Y_1C=0$. Now, if $C$ is chosen in a way that it is non-singular, then $b = n =0$. Also, if $D$ is non-singular, then $e=h=0$. These equations together give the complete solutions for the system of Yang-Baxter-like matrix equations for the given $A$ and $B$. 
\end{Example}
\section{Solutions when coefficient matrices are of size 2}
In this section, we discuss explicit solutions of the system of equations in \eqref{main_equation} when $A$ and $B$ are $2 \times 2$ matrices given in the Jordan-canonical form. We find solutions using the Gr$\ddot{o}$bner basis techniques.
\subsection{When A and B are diagonalizable}
Let both $A$ and $B$ be diagonalizable and are in Jordan form. Say,
\[ A = \begin{pmatrix}
 a & 0\\
 0 & b
\end{pmatrix}, \text{  } B = \begin{pmatrix}
 c & 0\\
 0 & d
\end{pmatrix}, \text{ and } X = \begin{pmatrix}
 x_1 & x_2\\
 x_3 & x_4
\end{pmatrix}.
\]
Then, solutions to the equations $AXA=XBX$ and $BXB = XAX$, equivalently can be read as zeros of a system of 8 polynomials, $f_1,f_2,...,f_8$.
\begin{proposition}
When $A,B$ assume the above mentioned form, and $a,b,c,d \neq 0$, then any non-trivial solution of \eqref{main_equation} have the following form:\\
\begin{itemize}
    \item[i)] If $x_1=0$, then
$X = \begin{pmatrix}
0& \alpha\\
0& \frac{b^2}{d}
\end{pmatrix}$ or $\begin{pmatrix}
0& 0\\
\alpha& \frac{b^2}{d}
\end{pmatrix}$, where $\alpha \in \mathbb{K}, a=b \text{ and } b^3=d^3$.
\item[ii)] If $x_2 = 0$, then $X = \begin{pmatrix}
\frac{a^2}{c}& 0\\
\alpha& \frac{b^2}{d}
\end{pmatrix}$ or $\begin{pmatrix}
\frac{a^2}{c}& 0\\
0& \frac{b^2}{d}
\end{pmatrix}$, where $ a^3=c^3, b^3=d^3, ab= a^2+b^2, \text{ and } cd = c^2+d^2$.
\item[iii)] If $x_3 = 0$, then $X = \begin{pmatrix}
\frac{a^2}{c}& \alpha\\
0& \frac{b^2}{d}
\end{pmatrix}$ or $\begin{pmatrix}
\frac{a^2}{c}& 0\\
0& \frac{b^2}{d}
\end{pmatrix}$, where $ a^3=c^3, b^3=d^3, ab= a^2+b^2, \text{ and } cd = c^2+d^2$.
\item[iv)] If $x_4 = 0$, then $X = \begin{pmatrix}
\frac{a^2}{c}& \alpha\\
0&0
\end{pmatrix}$ or $\begin{pmatrix}
\frac{a^2}{c}& 0\\
\alpha& 0
\end{pmatrix}$, where $\alpha \in \mathbb{K}, a^2b=c^2d$.
\item[v)] Any other non-singular solution $X$ of \eqref{main_equation}, has the form $X^2\ =\ U \begin{pmatrix}
 ac&0\\
 0&bd
\end{pmatrix} U^{-1}$, for some $U \in GL_2(\mathbb{K})$.
\end{itemize}
\end{proposition}
\begin{proof}
Equations in \eqref{main_equation} give the following polynomials, and a solution to it is a solution for this system of polynomials.
\begin{equation}\label{eq11}
\begin{aligned}
    f_1 &= a^2x_1-cx_1^2-dx_2x_3 &&
    f_2 &=&~ abx_2- cx_1x_2-dx_2x_4\\
    f_3 &= abx_3-cx_1x_3-dx_3x_4 &&
    f_4 &=&~ b^2x_4-cx_2x_3-dx_4^2\\
    f_5 &= c^2x_1-ax_1^2-bx_2x_3 &&
    f_6 &=&~ cdx_2-ax_1x_2-bx_2x_4\\
    f_7 &= cdx_3-ax_1x_3-bx_3x_4 &&
    f_8 &=&~ d^2x_4 -ax_2x_3-bx_4^2
\end{aligned}
\end{equation}
Now we can find a Gr$\ddot{o}$bner basis for the ideal  $\langle f_1,f_2,...,f_8 \rangle$ in $\mathbb{C}[a,b,c,d,x_1,...,x_4]$ with respect to the lexicographic order for $a>b>c>d>x_1>...>x_4$. But Gr$\ddot{o}$bner basis for this in general is more complicated and does not help to solve the equations. So, we have taken some additional assumptions such as $x_1 =0$, or $x_2=0$, or $x_3 =0$, o  $x_4=0$, and solved the system in each of the cases using Gr$\ddot{o}$bner basis. This gives solutions given in $(i)$ to $ iv)$. As the Gr$\ddot{o}$bner basis is very large, we have not included that here. Now, any other non-singular solution have the form $X^2\ =\ U \begin{pmatrix}
 ac&0\\
 0&bd
\end{pmatrix} U^{-1}$, for some $U \in GL_2(\mathbb{K})$, as $X^2$ is similar to $AB$ in that case .
\end{proof}
\begin{proposition}
\begin{itemize}
    \item[i)]When $a=0, b,c,d \neq 0$, then the only non-trivial solution to \eqref{main_equation} is $X= \begin{pmatrix}
0&0\\
0& \frac{d^2}{b}
\end{pmatrix}$, where $b^3 = d^3$.
\item[ii)] Similarly, when $b=0, a,c,d \neq 0$, then the only non-trivial solution is $X = \begin{pmatrix}
\frac{c^2}{a}&0\\
0& 0
\end{pmatrix}.$
\end{itemize}
\end{proposition}
\begin{proof}
when $a = 0$, we can rewrite the equations in \eqref{eq11}, and calculate the Gr$\ddot{o}$bner basis, which is simpler than the previous case. The corresponding generators of the ideals are the following. 
\begin{equation}
    \begin{aligned}
        f_1 &= cx_1^2+dx_2x_3, && f_2 &=& ~cx_1x_2+dx_2x_4\\
        f_3 &= cx_1x_3+dx_3x_4, && f_4 &=&~ b^2x_4-cx_2x_3-dx_4^2\\
        f_5 &= c^2x_1 -bx_2x_3, && f_6 &=&~ cdx_2-bx_2x_4\\
        f_7 &= cdx_3 -bx_3x_4, && f_8 &=& ~d^2x_4-bx_4^2.
    \end{aligned}
\end{equation}
Similar way, we can rewrite the equations in \eqref{eq11} when $b=0$, and $a,c,d\ \neq 0$, and calculate the Gr$\ddot{o}$bner basis, which leads to our solution.
\end{proof}
\begin{proposition}
When $a,c =0$ and $b,d \neq 0$, the solution has the following form.
$X = \begin{pmatrix}
\alpha &0\\
\beta & 0
\end{pmatrix}$ or $ \begin{pmatrix}
\alpha & \beta\\
0 & 0
\end{pmatrix}$, or  $\begin{pmatrix}
\alpha &0\\
0 & \frac{b^2}{d}
\end{pmatrix}$ where $\alpha, \beta \in \mathbb{K}$ and $b^3= d^3$.
\end{proposition}
\begin{proof}
When $a,c =0$ and $b,d \neq 0$, the system of equations has a much simpler form, and easy to find solutions. The corresponding equations are the following.
\begin{equation}
    \begin{aligned}
        f_1 &= dx_2x_3, && f_2 &=&~ dx_2x_4\\
        f_3 &= dx_3x_4, && f_4 &=&~ b^2x_4-dx_4^2\\
        f_5 &= bx_2x_3, && f_6 &=&~ bx_2x_4\\
        f_7 &= bx_3x_4, && f_8 &=&~ d^2x_4-bx_4^2.
    \end{aligned}
\end{equation}
Then the corresponding Gr$\ddot{o}$bner basis for the ideal generated by the above polynomials is,
\[\{x_3x_4d,\ x_2x_4d,\ x_2x_3d,\ x_4b^2-x_4^2d,\ x_4^2b-x_4d^2,\ x_3x_4b,\ x_2x_4b,\ x_2x_3b,\ x_4^3d-x_4bd^2\}\] 
On finding zeros of these, we get the solution forms as above.
\end{proof}
\subsection{When A is non-diagonalizable and B is diagonalizable}
Let $A\ = \ \begin{pmatrix}
 a&1\\
0&a 
\end{pmatrix}$ and $B\ =\ \begin{pmatrix}
 b&0\\
 0&c
\end{pmatrix}$, and $X$ as previously $\begin{pmatrix}
 x_1 & x_2\\
 x_3 & x_4
\end{pmatrix}$.
\begin{proposition}
\begin{itemize}
    \item[a)] When $a=0$, for a solution to \eqref{main_equation}, we have the following forms.
\begin{itemize}
    \item[i] When $b,c \neq 0$, \eqref{main_equation} has only trivial solution.
    \item[ii] If $b=0, c \neq 0$, then the solutions to \eqref{main_equation}  are $X\ =\ \begin{pmatrix}
     \alpha&\beta\\
     0&0
    \end{pmatrix}$ , where $ \alpha, \beta, \in \mathbb{K}$.
    \item[iii] If $b \neq 0, c =0$, then the solutions to \eqref{main_equation}  are $X\ =\ \begin{pmatrix}
     0&\alpha\\
     0&\beta
    \end{pmatrix}$ , where $ \alpha, \beta, \in \mathbb{K}$.
\end{itemize}
\item[b)] When $a \neq 0$, and $b, c\ \in\ \mathbb{K}$, \eqref{main_equation} has only trivial solution.
\end{itemize}
\end{proposition}
\begin{proof}
When $A\ =\ \begin{pmatrix}
 0&1\\
0&0 
\end{pmatrix}$, $B\ =\ \begin{pmatrix}
 b&0\\
 0&c
\end{pmatrix}$, and $X\ =\ \begin{pmatrix}
 x_1&x_1\\
 x_3&x_4
\end{pmatrix}$, the initial computation of Gr$\ddot{o}$bner basis corresponds to the ideal generated by the component polynomial equations from $AXA\ =\ XBX$ and $BXB\ =\ XAX$, from which one infers $x_3 = 0$. After modifying $X$ as $\begin{pmatrix}
 x_1&x_2\\
 0&x_4
\end{pmatrix}$, we have the Gr$\ddot{o}$bner basis as,
\[ \{ x_4c^2,\ x_2bc-x_1x_4,\ x_4^2c,\ x_1b^2,\ x_1x_2b+x_2x_4c,\ x_1^2b,\ x_1x_4^2,\ x_1^2x_4, x_1x_2x_4c\}.\]
Then, solving from Gr$\ddot{o}$bner basis gives us the solutions in $(a)$.\\
Now, when $A\ =\ \begin{pmatrix}
 a&1\\
 0&a
\end{pmatrix}$, and $B\ =\ \begin{pmatrix}
 b&0\\
 0&c
\end{pmatrix}$, such that $a \neq 0$, the initial Gr$\ddot{o}$bner basis of the corresponding ideal shows that $x_3 = 0$, irrespective the value of  $b,\ c$. Then on next iteration, after giving $x_3\ = 0$, we get $x_1,\ x_2,\ x_4\ =\ 0$. 
\end{proof}
\subsection{When A and B are non-diagonalizable}
Let $A\ = \ \begin{pmatrix}
 a&1\\
0&a 
\end{pmatrix}$, $B\ =\ \begin{pmatrix}
 b&1\\
 0&b
\end{pmatrix}$, and $X$ as previously $\begin{pmatrix}
  _1 & x_2\\
 x_3 & x_4
\end{pmatrix}$.
\begin{proposition}
Let $A$ and $B$ be as above. Then for any non-trivial solution X of \eqref{main_equation}, we have,
\begin{itemize}
    \item[i-] When $a,b \neq 0$, $X$ has one of the following form.\\
    $\begin{pmatrix}
a&1\\
0&a
\end{pmatrix}, 
\begin{pmatrix}
a+a\sqrt{\alpha}& \alpha\\
-a^2 & a-a \sqrt{\alpha}
\end{pmatrix} or 
\begin{pmatrix}
a-a\sqrt{\alpha}& \alpha\\
-a^2 & a+a \sqrt{\alpha}
\end{pmatrix}$ where $\alpha \in K$, and $a =\ b$, in that case. 
\item[ii-] When $a,b \neq 0$, and $a \neq b$, then any non-trivial solution $X$ has the form 
$\begin{pmatrix}
\alpha & \frac{b^2+ \alpha^2 - \alpha(a+b)}{ab}\\
-ab & (a+b)-\alpha
\end{pmatrix}$.
\item[iii-] When one of $a,b$ is zero, then \eqref{main_equation} has only trivial solution.
\end{itemize}
\end{proposition}
\begin{proof}
For the given $A,B$, and $X$ as $\begin{pmatrix}
 x_1 & x_2\\
 x_3 & x_4
\end{pmatrix}$, the initial computation of Gr$\ddot{o}$bner basis for the ideal generated by the polynomials appeared in \eqref{main_equation} shows that $x_3(a^2-b^2) = 0$. Then either $x_3 =0$ or $a^2 = b^2$. If $a^2 = b^2$, and $a =b$, then the system of equations becomes single YBE, and authors have found its solution in \cite{first}, which are given in $(i)$ above. Now, if $a \neq b$, but $a^2 = b^2$, we have $b^3+x_3a=0$ in Gr$\ddot{o}$bner basis. This gives $x_3 = -ab$. After substituting the said value for $x_3$, on the next iteration, we have $x_1ab^2+x_4ab^2-ab^3-b^4 = 0$ in the Gr$\ddot{o}$bner basis. This gives $x_1 +x_4 = a+b$. Now, adding $x_1 +x_4 -(a+b)$ to the basis, the Gr$\ddot{o}$bner basis becomes:
\[\{x_1+x_4-a-b,\ a^2-b^2,\ x_2ab-x_4^2+x_4a+x_4b-b^2,\ x_2b^3-x_4^2a+x_4ab+x_4b^2-ab^2\}.\]
Solving this gives the solution in $(ii)$.\\
When $a \neq 0$ and $b =0$, the Gr$\ddot{o}$bner basis correspond to the system of equations i  \eqref{main_equation} which gives, $x_3^2=0, x_4^3a=0$. Then we have $x_3,x_4=0$. Adding this to the basis, Gr$\ddot{o}$bner basis reduce to $\{x_2a^2+x_1a,\ x_1a^2,\ x_1x_2a,\ x_1^2a\}$. This gives us \eqref{main_equation} has only a trivial solution in this case.
\end{proof}
\textbf{Acknowledgment:}
 The first author would like to thank the University Grant Commission- Ministry of Human Resource Development, New Delhi, for the financial support provided through the CSIR-UGC fellowship. The second author is partially supported by CDRF, BITS Pilani, India, grant no. CDRF, C1/23/185. The third author is supported by the Ministry of Science, Technological Development and Innovations, Republic of Serbia, grant No. 451-03-66/2024-03/200029, and by the bilateral project between the Republic of Serbia and France (Generalized inverses on algebraic structures and applications), grant no. 337-00-93/2022-05/13. The authors are thankful to the anonymous referee for many helpful suggestions. 
\bibliographystyle{abbrv}
\bibliography{YB}

\begin{thebibliography}{10}

\bibitem{idempotent}
M.~S.~I. Adam, J.~Ding, Q.~Huang, and L.~Zhu.
\newblock Solving a class of quadratic matrix equations.
\newblock {\em Appl. Math. Lett.}, 82:58--63, 2018.

\bibitem{grobner2}
W.~W. Adams and P.~Loustaunau.
\newblock {\em An introduction to {G}r\"{o}bner bases}, volume~3 of {\em Graduate Studies in Mathematics}.
\newblock American Mathematical Society, Providence, RI, 1994.

\bibitem{Artin2}
E.~Artin.
\newblock Braids and permutations.
\newblock {\em Ann. of Math. (2)}, 48:643--649, 1947.

\bibitem{Artin}
E.~Artin.
\newblock Theory of braids.
\newblock {\em Ann. of Math. (2)}, 48:101--126, 1947.

\bibitem{Commuting-nilp-irr2}
V.~Baranovsky.
\newblock The variety of pairs of commuting nilpotent matrices is irreducible.
\newblock {\em Transform. Groups}, 6(1):3--8, 2001.

\bibitem{Commuting_nilp-irr}
R.~Basili.
\newblock On the irreducibility of commuting varieties of nilpotent matrices.
\newblock {\em J. Algebra}, 268(1):58--80, 2003.

\bibitem{BEMHMYA}
H.~Ben-El-Mechaieh and Y.~A. Mechaiekh.
\newblock An elementary proof of the brouwer's fixed point theorem.
\newblock {\em Arab. J. Math.}, pages 179--188, 2022.

\bibitem{set_theoretic}
M.~Castelli, G.~Pinto, and W.~Rump.
\newblock On the indecomposable involutive set-theoretic solutions of the {Y}ang-{B}axter equation of prime-power size.
\newblock {\em Comm. Algebra}, 48(5):1941--1955, 2020.

\bibitem{Diag2}
D.~Chen, Z.~Chen, and X.~Yong.
\newblock Explicit solutions of the {Y}ang-{B}axter-like matrix equation for a diagonalizable matrix with spectrum contained in {$\{1,\alpha,0\}$}.
\newblock {\em Appl. Math. Comput.}, 348:523--530, 2019.

\bibitem{diagonaliazable-simple}
D.~Chen and X.~Yong.
\newblock Finding solutions to the {Y}ang-{B}axter-like matrix equation for diagonalizable coefficient matrix.
\newblock {\em Symmetry}, 14(8), 2022.

\bibitem{grobner}
D.~A. Cox, J.~Little, and D.~O'Shea.
\newblock {\em Ideals, varieties, and algorithms}.
\newblock Undergraduate Texts in Mathematics. Springer, Cham, fourth edition, 2015.
\newblock An introduction to computational algebraic geometry and commutative algebra.

\bibitem{path}
N.~C. Dincic and B.~D. Djordjevic.
\newblock On the intrinsic structure of the solution set to the {Y}ang-{B}axter-like matrix equation.
\newblock {\em Revista de la Real Academia de Ciencias Exactas, F{\'\i}sicas y Naturales. Serie A: Matem{\'a}ticas (RACSAM)}, 116(2):20, 2022.

\bibitem{RLT}
N.~C. Dincic and B.~D. Djordjevic.
\newblock Yang-baxter-like matrix equation: a road less taken.
\newblock In {\em Matrix and Operator Equations (Mathematics Online First Collections)}, 2023.

\bibitem{DR}
J.~Ding and N.~H. Rhee.
\newblock A nontrivial solution to a stochastic matrix equation.
\newblock {\em East Asian Journal on Applied Mathematics}, 2(4):277--284, 2012.

\bibitem{BDD}
B.~D. Djordjević.
\newblock Doubly stochastic and permutation solutions to $axa = xax$ when $a$ is a permutation matrix.
\newblock {\em Linear Algebra and its Applications}, 661:79--105, 2023.

\bibitem{Drinfeld}
V.~G. Drinfeld.
\newblock On some unsolved problems in quantum group theory.
\newblock In {\em Quantum groups ({L}eningrad, 1990)}, volume 1510 of {\em Lecture Notes in Math.}, pages 1--8. Springer, Berlin, 1992.

\bibitem{M2}
D.~R. Grayson and M.~E. Stillman.
\newblock Macaulay2, a software system for research in algebraic geometry.
\newblock Available at \url{http://www.math.uiuc.edu/Macaulay2/}.

\bibitem{YBE-CS3}
H.~Huang, Z.~Huang, C.~Wu, C.~Jiang, D.~Fu, and C.~Lin.
\newblock Modified {N}ewton integration neural algorithm for solving time-varying {Y}ang-{B}axter-like matrix equation.
\newblock {\em Neural Processing Letters}, pages 1--15, 2022.

\bibitem{Diag}
Q.~Huang, M.~Saeed Ibrahim~Adam, J.~Ding, and L.~Zhu.
\newblock All non-commuting solutions of the {Y}ang-{B}axter matrix equation for a class of diagonalizable matrices.
\newblock {\em Oper. Matrices}, 13(1):187--195, 2019.

\bibitem{YBE-CS}
R.~Iordanescu, F.~F. Nichita, and I.~M. Nichita.
\newblock The {Y}ang-{B}axter equation,(quantum) computers and unifying theories.
\newblock {\em Axioms}, 3(4):360--368, 2014.

\bibitem{numerical}
A.~Kumar, J.~a.~R. Cardoso, and G.~Singh.
\newblock Explicit solutions of the singular {Y}ang-{B}axter-like matrix equation and their numerical computation.
\newblock {\em Mediterr. J. Math.}, 19(2):Paper No. 85, 19, 2022.

\bibitem{structure_set}
V.~Lebed and L.~Vendramin.
\newblock On structure groups of set-theoretic solutions to the {Y}ang-{B}axter equation.
\newblock {\em Proc. Edinb. Math. Soc. (2)}, 62(3):683--717, 2019.

\bibitem{manifold}
L.~Lu.
\newblock Manifold expressions of all solutions of the {Y}ang--{B}axter-like matrix equation for rank-one matrices.
\newblock {\em Applied Mathematics Letters}, 132:108175, 2022.

\bibitem{idempotent2}
S.~I.~A. Mansour, J.~Ding, and Q.~Huang.
\newblock Explicit solutions of the {Y}ang-{B}axter-like matrix equation for an idempotent matrix.
\newblock {\em Appl. Math. Lett.}, 63:71--76, 2017.

\bibitem{first}
H.~Mukherjee and A.~A. M.
\newblock Solutions to the matrix {Y}ang-{B}axter equation, 2022(In communication),arXiv:2209.04605.

\bibitem{Tridempotent}
M.~Saeed Ibrahim~Adam, J.~Ding, Q.~Huang, and L.~Zhu.
\newblock All solutions of the {Y}ang-{B}axter-like matrix equation when {$A^3=A$}.
\newblock {\em J. Appl. Anal. Comput.}, 9(3):1022--1031, 2019.

\bibitem{two_eigenvalue}
D.~Shen and M.~Wei.
\newblock All solutions of the {Y}ang-{B}axter-like matrix equation for diagonalizable coefficient matrix with two different eigenvalues.
\newblock {\em Appl. Math. Lett.}, 101:106048, 8, 2020.

\bibitem{sylvester-original}
J.~Sylvester.
\newblock Sur l'equations en matrices $px=xq$.
\newblock {\em C. R. Acad. Sci. Paris.}, 99(2):67–71, 1884.

\bibitem{rank1}
H.~Tian.
\newblock All solutions of the {Y}ang-{B}axter-like matrix equation for rank-one matrices.
\newblock {\em Appl. Math. Lett.}, 51:55--59, 2016.

\bibitem{YBE-CS2}
F.~A. Vind, A.~Foerster, I.~S. Oliveira, R.~S. Sarthour, D.~d.~O. Soares-Pinto, A.~M.~d. Souza, and I.~Roditi.
\newblock Experimental realization of the {Y}ang-{B}axter equation via nmr interferometry.
\newblock {\em Scientific reports}, 6(1):1--8, 2016.

\bibitem{gen.sylvester}
K.~wah Eric~Chu.
\newblock The solution of the matrix equations $\mathrm{AXB-CXD=E}$ and $\mathrm{(YA-DZ,YC-BZ)=(E,F)}$.
\newblock {\em Linear Algebra and its Applications}, 93:93--105, 1987.

\bibitem{rank2.2}
D.~Zhou, G.~Chen, and J.~Ding.
\newblock On the {Y}ang-{B}axter-like matrix equation for rank-two matrices.
\newblock {\em Open Math.}, 15(1):340--353, 2017.

\bibitem{rank2.1}
D.~Zhou, G.~Chen, and J.~Ding.
\newblock Solving the {Y}ang-{B}axter-like matrix equation for rank two matrices.
\newblock {\em J. Comput. Appl. Math}, 313(1):142--151, 2017.

\bibitem{ortho}
D.~Zhou, G.~Chen, J.~Ding, and H.~Tian.
\newblock Solving the {Y}ang-{B}axter-like matrix equation with non-diagonalizable elementary matrices.
\newblock {\em Commun. Math. Sci.}, 17(2):393--411, 2019.

\bibitem{3eigenvalue}
D.-M. Zhou, X.-X. Ye, Q.-W. Wang, J.-W. Ding, and W.-Y. Hu.
\newblock Explicit solutions of the {Y}ang-{B}axter-like matrix equation for a singular diagonalizable matrix with three distinct eigenvalues.
\newblock {\em Filomat}, 35(12):3971--3982, 2021.

\end{thebibliography}

\end{document}